\documentclass[11pt,twoside]{article}
\usepackage{amsmath,amssymb,graphics,graphicx}
\pdfoutput=1
\newtheorem{proposition}{Proposition}[section]
\newtheorem{theorem}[proposition]{Theorem}

\newtheorem{corollary}[proposition]{Corollary}

\newtheorem{remark}[proposition]{Remark}

\newenvironment{proof}{{\noindent \bf Proof:}}{\hfill $\fbox{}$ \vspace*{5mm}}

\DeclareMathOperator{\T}{T}
\renewcommand{\H}{\mathrm{H}}
\DeclareMathOperator{\diag}{diag}
\DeclareMathOperator{\Diag}{Diag}
\newcommand{\mvk}{\mathbb{V}_k}
\newcommand{\mpk}{{\mathbb{P}_k}}

\newcommand{\mpxi}{{\mathbb{P}_\xi}}
\newcommand{\wt}{\widetilde}
\newcommand{\bsmat}{\left[\begin{smallmatrix}}
\newcommand{\esmat}{\end{smallmatrix}\right]}

\begin{document}
\title{Perturbation Analysis of An Eigenvector-Dependent Nonlinear Eigenvalue Problem With Applications\thanks{This work was supported by NSFC No. 11671023, 11421101, 11671337, 11771188.}}
\author{Yunfeng Cai\thanks{LMAM \& DSEC, School of Mathematical Sciences, Peking Univ., Beijng, P.R. China,100871 (yfcai@math.pku.edu.cn).}
\and Zhigang Jia\thanks{School of Mathematics and Statistics, Jiangsu Normal University, Xuzhou, P.R. China, 221116 (zhgjia@jsnu.edu.cn). }
\and Zheng-Jian Bai\thanks{School of Mathematical Sciences and Fujian Provincial Key Laboratory on Mathematical Modeling \& High Performance Scientific Computing,  Xiamen University, Xiamen, P.R. China, 361005 (zjbai@xmu.edu.cn). The research of this author was partially supported by the Natural Science Foundation of Fujian Province of China (No. 2016J01035) and the Fundamental Research Funds for the Central Universities.}
}

\maketitle
\begin{abstract}
The eigenvector-dependent nonlinear eigenvalue problem (NEPv)  $A(P)V=V\Lambda$,
where the columns of  $V\in\mathbb{C}^{n\times k}$ are orthonormal,
$P=VV^{\H}$,  $A(P)$ is Hermitian, and $\Lambda=V^{\H}A(P)V$,
arises in many important applications, such as the discretized Kohn-Sham equation in electronic structure calculations and the trace ratio problem in linear discriminant analysis.
In this paper, we perform a perturbation analysis for the NEPv,
which gives upper bounds for the distance between the solution to the original NEPv and the solution to the perturbed NEPv. A condition number for the NEPv is introduced, which reveals the factors that affect the sensitivity of the solution.
Furthermore, two computable error bounds are given for the NEPv, which can be used to measure the quality of an approximate solution.
The theoretical results are validated by numerical experiments for the Kohn-Sham equation and the trace ratio optimization.
\end{abstract}

\vspace{3mm}


\vspace{3mm}

\section{Introduction}\label{sec:intro}
In this paper, we study the perturbation theory of  the following eigenvector-dependent nonlinear eigenvalue problem (NEPv)
\begin{equation}\label{eq:nep}
A(P)V=V\Lambda,
\end{equation}
where $V\in\mathbb{C}^{n\times k}$ has orthonormal column  vectors,
$P=VV^{\H}$,
$A(P)$ is a continuous Hermitian matrix-valued function of $P$, 
and $\Lambda=V^{\H}A(P)V\in\mathbb{C}^{k\times k}$ is Hermitian,
the eigenvalues of $\Lambda$ are also eigenvalues of $A(P)$.
Usually, in practical applications, $k\ll n$, and the eigenvalues of $\Lambda$ are the $k$ smallest or largest eigenvalues of $A(P)$.
In this paper, we restrict our discussions to the case of the $k$ smallest eigenvalues.
Furthermore, we consider $A(P)$ in the following form
\begin{equation}\label{eq:a012}
A(P)=A_0+A_1(P)+A_2(P),
\end{equation}
where $A_0$, $A_1(P)$ and $A_2(P)$ are all Hermitian,
$A_0\in\mathbb{C}^{n\times n}$ is a constant matrix,
$A_1(P)$ is a homogeneous linear function of $P$,
and $A_2(P)$ is a nonlinear function of $P$.

Notice that if $V$ is a solution \eqref{eq:nep},
then so is $VQ$ for any $k\times k$ unitary matrix $Q$.
Therefore, two solutions $V$, $\wt{V}$ are essentially the same if $\mathcal{R}(V)=\mathcal{R}(\wt{V})$,
where $\mathcal{R}(V)$ and $\mathcal{R}(\wt{V})$ are the subspaces spanned
by the column vectors of $V$ and $\wt{V}$, respectively.
Throughout the rest of this paper,
when we say  that $V$ is a solution to \eqref{eq:nep},
we mean that the class $\{VQ \;|\; Q^{\H}Q=I_k\}$ solves \eqref{eq:nep}.

Perhaps, the most well-known NEPv of the form \eqref{eq:nep} is
the discretized Kohn-Sham (KS) equation arising from density function theory in electronic structure calculations (see~\cite{chdg:14,mart:2004,sacs:2010} and references therein).
NEPv~\eqref{eq:nep} also arises from
the trace ratio optimization in the linear discriminant analysis
for dimension reduction \cite{ngbs:2010,zhli:2014a,zhli:2014b},
and the Gross-Pitaevskii equation for modeling particles in the state of matter called the Bose-Einstein condensate \cite{badu:06,jakm:14,jixx:16}.
We believe that more potential applications will emerge.

The most widely used method for solving NEPv \eqref{eq:nep} is the  so-called self-consistent field (SCF) iteration \cite{mart:2004,sacs:2010}.
Starting with orthonormal $V_0\in\mathbb{C}^{n\times k}$,
at the $l$th SCF iteration, one computes an orthonormal eigenvector matrix $V_l$ associated with the $k$ smallest eigenvalues of $A(V_{l-1}V_{l-1}^{\H})$,
and then $V_l$ is used as the approximation in the next iteration.
Convergence analysis of SCF iteration for the KS equation is studied in \cite{liww:14,liww:15,yamw:07},
for the trace ratio problem in \cite{zhli:2014b}.
Quite recently,
in \cite{cazb:2017a}, an existence and uniqueness condition of the solutions to NEPv \eqref{eq:nep} is given,
and the convergence of the SCF iteration is also studied.

In practical applications, $A(P)$ is usually obtained from the discretization of operators or constructed from empirical data,
thus, contaminated by errors and noises.
As a result, the NEPv  \eqref{eq:nep} to be solved is in fact a perturbed NEPv.
So, it is natural to ask whether  we can trust the approximate solution obtained by solving the perturbed NEPv via certain numerical methods, say the SCF iteration.
To be specific, let the perturbed NEPv be of the form
\begin{equation}\label{eq:nep2}
\wt{A}(\wt{P})\wt{V}=\wt{V}\wt{\Lambda},
\end{equation}
where $\wt{V}$ has orthonormal column vectors,  $\wt{P}=\wt{V}\wt{V}^{\H}$,
$\wt{\Lambda}=\wt{V}^{\H}\wt{A}(\wt{P})\wt{V}\in\mathbb{C}^{k\times k}$,
and
\begin{equation}\label{eq:wa012}
\wt{A}(\wt{P})=\wt{A}_0+\wt{A}_1(\wt{P})+\wt{A}_2(\wt{P})
\end{equation}
is a continuous Hermitian matrix-valued function of $\wt{P}$,
$\wt{A}_0$ is a constant Hermitian matrix,
$\wt{A}_1$ and $\wt{A}_2$ are perturbed functions of $A_1$ and $A_2$, respectively,
and $\wt{A}_1(\wt{P})$, $\wt{A}_2(\wt{P})$ are still Hermitian.
Assume that the original NEPv \eqref{eq:nep} has a solution $V_*$.
Then we need to answer the following two fundamental questions:\\
{\bf Q1.} Under what conditions the perturbed NEPv \eqref{eq:nep2} has a solution $\wt{V}_*$ nearby $V_*$? \\
{\bf Q2.} What's the distance between  $\mathcal{R}(V_*)$  and $\mathcal{R}(\wt{V}_*)$?\\

Let  $\mathcal{X}$ and $\mathcal{Y}$ be two $k$-dimensional subspaces of $\mathbb{C}^n$. Let the  columns of $X$ form an orthonormal basis for $\mathcal{X}$ and the  columns of $Y$ form an orthonormal basis for $\mathcal{Y}$.
We use $\|\sin\Theta(\mathcal{X},\mathcal{Y})\|_2$ to measure the distance between  $\mathcal{X}$ and $\mathcal{Y}$,
where
\begin{equation}\label{eq:mat-angles-XY}
\Theta(\mathcal{X},\mathcal{Y})=\diag(\theta_1(\mathcal{X},\mathcal{Y}),\ldots,\theta_k(\mathcal{X},\mathcal{Y})).
\end{equation}
Here, $\theta_j(\mathcal{X},\mathcal{Y})$'s denote the $k$ {\em canonical angles} between $\mathcal{X}$ and $\mathcal{Y}$ \cite[p. 43]{stsu:1990},
which can be defined as
\begin{equation}\label{eq:indv-angles-XY}
0\le\theta_j(\mathcal{X},\mathcal{Y}):=\arccos\sigma_j\le\frac {\pi}2\quad\mbox{for $1\le j\le k$},
\end{equation}
where $\sigma_j$'s are the singular values of $X^{\H}Y$.


In this paper, we will focus on {\bf Q1} and {\bf Q2}.
The results are established via two approaches.
One is based on the well-known $\sin\Theta$ theorem in the perturbation theory of Hermitian matrices \cite{daka:70} and Brouwer's fixed-point theorem \cite{kham:14};
The other is inspired by J.-G. Sun's technique (e.g., \cite{lin2002perturbation, sun1998perturbation, sun2003perturbation,sun2003perturbation2}) --
finding the radius of the perturbation by constructing an equation of the radius via the fixed-point theorem.
Two perturbation bounds can be obtained from these two approaches, and each of them has its own merits.
Based on the perturbation bounds, a condition number for the NEPv \eqref{eq:nep} is introduced,
which quantitatively reveals the factors that affect the sensitivity of the solution.
As corollaries, two computable error bounds are provided to measure the quality of the computed solution.
Theoretical results are validated by numerical experiments for the KS equation and the trace ratio optimization.

\medskip

The rest of this paper is organized as follows.
In section~\ref{sec:theory}, we use two approaches to answer {\bf Q1} and {\bf Q2},
followed by some discussions on the condition number and error bounds for NEPv \eqref{eq:nep}.
In section~\ref{sec:appl}, we apply our theoretical results to the KS equation and the trace ratio optimization problem, respectively.
Finally, we give our concluding remarks in section~\ref{sec:conclusion}.

\section{Main results}\label{sec:theory}
In this section we provide two approaches to answer {\bf Q1} and {\bf Q2}. A condition number and error bounds for NEPv  will also be discussed.
Before we proceed, we introduce the following notation,
which will be used throughout the rest of this paper.

$\mathbb{C}^{n\times m}$ stands for the set of all $n\times m$  matrices with complex entries.
The superscripts ``${\cdot}^{\T}$'' and ``${\cdot}^{\H}$'' take the transpose and the complex
conjugate transpose of a matrix or vector, respectively.
The symbol $\|\cdot\|_2$ denotes the 2-norm of a matrix or vector.
Unless otherwise specified, we denote by $\lambda_j(H)$ for $1\le j\le n$
the eigenvalues of a Hermitian matrix $H\in\mathbb{C}^{n\times n}$ and they are always arranged in nondecreasing order:
$\lambda_1(H)\le\lambda_2(H)\le\cdots\le\lambda_n(H)$.
Define
\begin{subequations}
\begin{align}
\mvk&:=\{V\in\mathbb{C}^{n\times k}\; \big| \; V^{\H}V=I_k\},\\
\mpk&:=\{P\in\mathbb{C}^{n\times n}\; \big| \; P=VV^{\H}, V\in\mvk\}.
\end{align}
\end{subequations}

Let $V_*$, $\wt{V}_*\in\mvk$ be the solutions to \eqref{eq:nep} and \eqref{eq:nep2}, respectively.
For any $\xi>0$, define
\begin{align}\label{vxi}
\mathbb{V}_\xi&:=\{V\in\mathbb{C}^{n\times k}\; \big| \; V^{\H}V=I_k, \|\sin\Theta(\mathcal{R}(V), \mathcal{R}(V_*))\|_2\le \xi\},\\
\mathbb{P}_\xi&:=\{P\in\mathbb{C}^{n\times n}\; \big| \; P=VV^{\H}, V\in\mathbb{V}_\xi\}.
\end{align}
Denote $P_*=V_*V_*^{\H}$, $\wt{P}_*=\wt{V}_*\wt{V}_*^{\H}$, $\Delta A_0 = \wt{A}_0 - A_0$,
and also
\begin{subequations}\label{ad}
\begin{align}
\delta_0 &= \|\wt{A}_0 - A_0\|_2,& &\\
\delta_1 &= \sup_{{P\in\mpxi}} {\|\wt{A}_1(P) - A_1(P)\|_2},&
d_1 &=\sup_{P\ne P_*, { P\in\mpxi}}  \frac{\|A_1(P) - A_1(P_*)\|_2}{\|P - P_*\|_2},
\\
\delta_2 &= \sup_{{P\in\mpxi}} {\|\wt{A}_2(P) - A_2(P)\|_2},&
d_2 &= \sup_{P\ne P_*, { P\in\mpxi}}  \frac{\|A_2({P}) - A_2(P_*)\|_2}{\|{P} - P_*\|_2},
\\
\delta & = \delta_0+\delta_1+\delta_2, & d&= d_1+d_2.
\end{align}
\end{subequations}
Note here that $\delta$ can be used to measure the magnitude of the perturbation, and
$d$ is a ``local Lipschitz constant'' such that
\begin{equation}
\|A(P)-A(P_*)\|_2\le d\|P-P_*\|_2
\end{equation}
for all $P\in\mathbb{P}_{\xi}$. Thus, we may use $d$ to measure the sensitivity of $A(P)$ within $\mpxi$.

\subsection{Approach one}
In this subsection, we use
the famous Weyl  Theorem \cite[p.203]{stsu:1990},
Davis-Kahan $\sin\Theta$ theorem \cite{daka:70},
and Brouwer's fixed-point theorem \cite{kham:14}
to answer questions {\bf Q1} and {\bf Q2}.

\begin{theorem}\label{thm1}
Let $V_*\in\mvk$ be a solution to \eqref{eq:nep}, $P_*=V_*V_*^{\H}$,
and
\begin{equation}\label{g}
g=\lambda_{k+1}(A(P_*))-\lambda_k(A(P_*))>0.
\end{equation}
If
\begin{equation}\label{con2}
\delta< \frac12 \ g -d,
\end{equation}
then the perturbed NEPv \eqref{eq:nep2} has a solution $\wt{V}_*\in\mathbb{V}_{\xi_*}$ with
\begin{align}\label{xistar}
\xi_*=\frac{2\delta}{g - d - \delta + \sqrt{(g - d - \delta)^2-4d\delta}}.
\end{align}
\end{theorem}

\begin{proof}
Using \eqref{con2}, we know that $\xi_*$ given by \eqref{xistar} is a positive constant.
Then it is easy to see that $\mathbb{P}_{\xi_*}$ is a nonempty  bounded  closed convex set in $\mathbb{C}^{n\times k}$.
For any  $\wt{V}\in\mathbb{V}_{\xi_*}$,  letting $\wt{P}=\wt{V}\wt{V}^{\H}$,
we define $\phi(\wt{P})=\wt{P}_\phi=\wt{V}_\phi\wt{V}_\phi^{\H}$ for $\wt{V}_\phi=[\tilde{v}_{\phi 1},\dots,\tilde{v}_{\phi k}]$,
where $\tilde{v}_{\phi j}$ is an eigenvector of $\wt{A}(\wt{P})$ corresponding with $\lambda_j(\wt{A}(\wt{P}))$ for $j=1,\dots, k$
and $\phi(\wt{P})\in\mathbb{P}_{\xi_*}$.
If we can show that
\begin{itemize}
\item[(a)]
$\lambda_{k+1}(\wt{A}(\wt{P}))-\lambda_k(\wt{A}(\wt{P}))>0$
(which implies that the mapping $\phi(\cdot)$ is well-defined in the sense that  $\phi(\wt{P})$ is unique);
\item[(b)]
$\phi(\cdot)$ is a continuous mapping within $\mathbb{P}_{\xi_*}$;
\item[(c)]
$\phi(\wt{P})\in\mathbb{P}_{\xi_*}$,
\end{itemize}
then by Brouwer's fixed-point theorem \cite{kham:14},
$\phi(\wt{P})$ has a fixed point in $\mathbb{P}_{\xi_*}$.
Let $\wt{P}_*=\wt{V}_*\wt{V}_*^{\H}$ be  the fixed point, where $\wt{V}_*\in\mathbb{V}_{\xi_*}$. Then $\wt{V}_*$ is a solution to the perturbed NEPv \eqref{eq:nep2}.
Hence the conclusion follows immediately.
Next, we show $(a)$, $(b)$ and $(c)$ in order.

\medskip

\noindent{\em Proof of $(a)$}
First, using \eqref{con2} and \eqref{xistar}, we have
\begin{align}
\xi_*&<\frac{2\delta}{g - d - \delta + \sqrt{( d + \delta)^2-4d\delta}}\notag\\
&=\frac{2\delta}{g - d - \delta + | d - \delta|}\notag\\
&=\left\{\begin{array}{cc}
\frac{2\delta}{g  -2 \delta }, & \mbox{if } d\ge \delta,\\
\frac{2\delta}{g - 2d }, & \mbox{otherwise}
\end{array}\right.\notag\\
&<1.\label{xil}
\end{align}

Second, direct calculations give rise to
\begin{subequations}\label{da}
\begin{align}
\|\wt{A}(\wt{P})-A(P_*)\|_2
&\le \|\wt{A}_0-A_0\|_2+ \|\wt{A}_1(\wt{P})-A_1(P_*)\|_2+ \|\wt{A}_2(\wt{P})-A_2(P_*)\|_2\notag\\
&\le \delta_0 +  \|\wt{A}_1(\wt{P})-{A}_1(\wt{P})\|_2+ \|{A}_1(\wt{P})-A_1(P_*)\|_2 \notag\\
&\mbox{}\qquad\quad \, + \|\wt{A}_2(\wt{P})-{A}_2(\wt{P})\|_2+ \|{A}_2(\wt{P})-A_2(P_*)\|_2\notag\\
&\le \delta+ d\|\wt{P}-P_*\|_2\label{da1}\\
&\le \delta+d\xi_*,\label{da2}
\end{align}
\end{subequations}
where \eqref{da1} uses \eqref{ad}, \eqref{da2} uses
$\|\wt{P}-P_*\|_2=\|\sin\Theta(\mathcal{R}(V_*), \mathcal{R}(\wt{V}))\|_2$ and $\wt{V}\in \mathbb{V}_{\xi_*}$.

Third, by  the famous Weyl Theorem \cite[p.203]{stsu:1990}, we have
\begin{align}\label{weyl}
|\lambda_j(\wt{A}(\wt{P}))-\lambda_j(A(P_*))|\le \|\wt{A}(\wt{P})-A(P_*)\|_2,
\ \mbox{for } j=1,2,\dots, n.
\end{align}
Then it follows that
\begin{subequations}\label{lamd}
\begin{align}
&\mbox{}\quad\ \lambda_{k+1}(\wt{A}(\wt{P}))-\lambda_k(\wt{A}(\wt{P}))\notag\\
&=g+[\lambda_{k+1}(\wt{A}(\wt{P})) - \lambda_{k+1}(A(P_*))] +[\lambda_k(A(P_*))- \lambda_k(\wt{A}(\wt{P}))]\notag\\
&\ge g - 2\|\wt{A}(\wt{P})-A(P_*)\|_2\label{lamd1}\\
&\ge g - 2\delta -2 d\xi_*\label{lamd2}\\
&>0,\label{lamd3}
\end{align}
\end{subequations}
where \eqref{lamd1} uses  \eqref{weyl},
\eqref{lamd2} uses \eqref{da},
\eqref{lamd3} uses \eqref{xil} and \eqref{con2}.

\medskip

\noindent{\em Proof of $(b)$}
We verify that $\phi(\cdot)$ is a continuous mapping  within $\mathbb{P}_{\xi_*}$ by showing that for any $\wt{V}_1$, $\wt{V}_2\in\mathbb{V}_{\xi_*}$,
$\|\phi(\wt{P}_1)-\phi(\wt{P}_2)\|_2\rightarrow 0$ as $\|\wt{P}_1-\wt{P}_2\|_2\rightarrow 0$, where
$\wt{P}_1=\wt{V}_1\wt{V}_1^{\H}$ and $\wt{P}_2=\wt{V}_2\wt{V}_2^{\H}$.

Let $\phi(\wt{P}_1)=\wt{V}_{1\phi}\wt{V}_{1\phi}^{\H}$, $\phi(\wt{P}_2)=\wt{V}_{2\phi}\wt{V}_{2\phi}^{\H}$, and
\begin{align*}
\wt{R}=\wt{A}(\wt{P}_1)\wt{V}_{2\phi} - \wt{V}_{2\phi}\diag(\lambda_1(\wt{A}(\wt{P}_2)),\dots,\lambda_k(\wt{A}(\wt{P}_2))).
\end{align*}
Then
\begin{align*}
\wt{R}=[\wt{A}(\wt{P}_1) - \wt{A}(\wt{P}_2)]\wt{V}_{2\phi},
\end{align*}
and hence
\begin{align*}
&\|\wt{R}\|_2=\|[\wt{A}(\wt{P}_1)-\wt{A}(\wt{P}_2)] \wt{V}_{2\phi}\|_2\le \|\wt{A}(\wt{P}_1)-\wt{A}(\wt{P}_2)\|_2.
\end{align*}
Using  \eqref{xil}--\eqref{weyl}, we have
\begin{align}
&\mbox{}\quad \ \lambda_{k+1}(\wt{A}(\wt{P}_2)) - \lambda_{k}(\wt{A}(\wt{P}_1))\notag\\
&=g +[\lambda_{k+1}(\wt{A}(\wt{P}_2))-\lambda_{k+1}(A(P_*))] - [\lambda_{k}(\wt{A}(\wt{P}_1))-\lambda_k(A(P_*))]\notag\\
&\ge g - 2(\delta+d\xi_*)
\ge g-2(\delta+d)>0.\label{gg}
\end{align}
By Davis-Kahan $\sin\Theta$ theorem \cite{daka:70}, we have
\begin{align}\label{sin12}
\|\sin\Theta(\mathcal{R}(\wt{V}_{1\phi}), \mathcal{R}(\wt{V}_{2\phi}))\|_2\le \frac{\|\wt{R}\|_2}{\lambda_{k+1}(\wt{A}(\wt{P}_2)) - \lambda_{k}(\wt{A}(\wt{P}_1))}.
\end{align}
Letting $\|\wt{P}_1-\wt{P}_2\|_2\rightarrow 0$, we know that
$\|\wt{R}\|_2\rightarrow 0$  since $\wt{A}(\cdot)$ is continuous.
Then it follows from \eqref{gg} and \eqref{sin12} that
\begin{align*}
\|\phi(\wt{P}_1)-\phi(\wt{P}_2)\|_2=\|\sin\Theta(\mathcal{R}(\wt{V}_{1\phi}), \mathcal{R}(\wt{V}_{2\phi}))\|_2 \le \frac{\|\wt{R}\|_2}{g-2(\delta+d)}\rightarrow 0.
\end{align*}
Therefore, $\|\phi(\wt{P}_1)-\phi(\wt{P}_2)\|_2\rightarrow 0$.

\medskip

\noindent{\em Proof of $(c)$}
Define
\[
R=\wt{A}(\wt{P})V_* - V_*\Lambda_*,
\]
where $\Lambda_*=V_*^{\H}A(P_*)V_*$.
Then
\begin{align}\label{R}
R=[\wt{A}(\wt{P})-{A}(P_*)]V_*.
\end{align}
Using \eqref{da} and \eqref{weyl}, we  have
\begin{align}
\lambda_{k+1}(A(P_*)) - \lambda_k(\wt{A}(\wt{P}))
&=\lambda_{k+1}(A(P_*))-{\lambda}_k(A(P_*))+\lambda_{k}(A(P_*))-\lambda_k(\wt{A}(\wt{P}))\notag\\
&\ge g - \delta - d\xi_*>0.\label{lamdif}
\end{align}
Then it follows that
\begin{subequations}\label{sin}
\begin{align}
&\mbox{}\quad\|P_*-\phi(\wt{P}))\|_2=\|\sin\Theta(\mathcal{R}(V_*),\mathcal{R}(\wt{V}))\|_2\notag\\
&\le \frac{\|R\|_2}{\lambda_{k+1}(A(P_*))-{\lambda}_k(\wt{A}(\wt{P}))}\label{sin1}\\
&\le  \frac{\|\wt{A}(\wt{P})-A(P_*)\|_2}{g-\delta-d\xi_*}\label{sin2}\\
&\le  \frac{\delta + d\xi_*}{g - \delta - d\xi_*}\label{sin3}\\
&=\xi_*,\label{sin4}
\end{align}
\end{subequations}
where \eqref{sin1} uses Davis-Kahan $\sin\Theta$ theorem \cite{daka:70},
\eqref{sin2} uses \eqref{R} and \eqref{lamdif},
\eqref{sin3} uses \eqref{da},
\eqref{sin4} uses  \eqref{xistar}.
Therefore, $\phi(\wt{P})\in\mathbb{P}_{\xi_*}$.
This completes the proof.
\end{proof}

\begin{remark}{\rm
The above approach is inspired by \cite{zhang2013perturbation} and is also used in \cite{cazb:2017a},
where the existence and uniqueness of the solution to \eqref{eq:nep} and the convergence of the SCF iteration are studied.
}\end{remark}

\subsection{Approach two}
In this subsection, we use another approach to answer questions {\bf Q1} and {\bf Q2},
which is inspired by J.-G. Sun's technique, see e.g., \cite{lin2002perturbation, sun1998perturbation, sun2003perturbation,sun2003perturbation2}.

\begin{theorem}\label{thm2}
Let $V_*\in\mvk$ be a solution to \eqref{eq:nep}, $P_*=V_*V_*^{\H}$, $g$ be given by \eqref{g},
and
\begin{equation}\label{gzeta}
h=\max_{1\le j\le k} [\lambda_{k+j}(A(P_*))  -  \lambda_j(A(P_*))],
\qquad
\zeta=\frac{\sqrt{g}}{\sqrt{g}+\sqrt{2h}}.
\end{equation}
Assume that $\delta$ is sufficiently small such that
\begin{equation}\label{feta}
f(\eta)\equiv g \eta - d\eta\sqrt{1+\eta^2} - (1+\eta^2)\delta = 0
\end{equation}
has positive roots,
and its smallest positive root, denoted by $\eta_*$, is smaller than $\zeta$.
Then the perturbed NEPv \eqref{eq:nep2} has a solution $\wt{V}_*\in\mathbb{V}_{\tau_*}$
with
\begin{equation}\label{tau}
\tau_*=\frac{\eta_*}{\sqrt{1+\eta_*^2}}.
\end{equation}
\end{theorem}

\begin{proof}
Let $[V_*, V_c]$ be a unitary matrix such that
\begin{align}\label{vav}
[V_*, V_c]^{\H}A(P_*)[V_*, V_c]=\begin{bmatrix}\Lambda_* & 0 \\ 0 & \Lambda_c\end{bmatrix},
\end{align}
where
\begin{align*}
\Lambda_*=\diag(\lambda_1(A(P_*)),\dots,\lambda_k(A(P_*))),\quad
\Lambda_c=\diag(\lambda_{k+1}(A(P_*)),\dots,\lambda_n(A(P_*))).
\end{align*}
Then that the perturbed NEPv \eqref{eq:nep2} has a solution $\wt{V}_*$
is equivalent to that there exists a unitary matrix $[\wt{V}_*, \wt{V}_c]$ such that
\begin{align}\label{vav2}
[\wt{V}_*, \wt{V}_c]^{\H}\wt{A}(\wt{P}_*)[\wt{V}_*, \wt{V}_c]=\begin{bmatrix}\wt{\Lambda}_* &  0 \\ 0 & \wt{\Lambda}_c\end{bmatrix},
\end{align}
where $\wt{\Lambda}_*$ is Hermitian and its eigenvalues are the $k$ smallest eigenvalues of $\wt{A}(\wt{P}_*)$.

Without loss of generality\footnote{Note that $k\ll n$ and thus $2k\le n$. By the CS decomposition \cite[Chapter 1, Theorem 5.1]{stsu:1990}, we know that there exist unitary matrices $\diag(U_1,U_2)$ and $\diag(U_3,U_4)$ such that
$[\wt{V}_*, \wt{V}_c]=[V_*, V_c] \diag(U_1, U_2)\begin{bmatrix}\Gamma & -\Sigma & 0\\ \Sigma & \Gamma & 0 \\ 0 & 0 & I\end{bmatrix}\diag(U_3,U_4)^{\H}$.
Rewrite $[\wt{V}_*, \wt{V}_c]=[\wt{V}_*, \wt{V}_c] \diag(Q_*^{\H}U_3U_1^{\H}Q_*, Q_c^{\H}U_4U_2^{\H}Q_c)$.
It still holds \eqref{vav2}.
Then \eqref{vvz} follows immediately
by setting $Z=U_2\bsmat \Sigma\Gamma^{-1}\\ 0\esmat U_1^{\H}$.},
 we let
\begin{align}\label{vvz}
[\wt{V}_*, \wt{V}_c]=[V_*, V_c]\begin{bmatrix}I_k & -Z^{\H} \\ Z & I_{n-k}\end{bmatrix}\begin{bmatrix}(I_k+Z^{\H}Z)^{-\frac12} & 0 \\ 0 & (I_{n-k}+ZZ^{\H})^{-\frac12}\end{bmatrix}\diag(Q_*,Q_c),
\end{align}
where $Z\in\mathbb{C}^{(n-k)\times k}$ is a parameter matrix,
 $Q_*\in\mathbb{C}^{k\times k}$ and $Q_c\in\mathbb{C}^{(n-k)\times (n-k)}$ are arbitrary unitary matrices.
Substituting \eqref{vvz} into  \eqref{vav2}, 
we get
\begin{subequations}
\begin{align}
&(I_k+Z^{\H}Z)^{-\frac12}[I_k,Z^{\H}] D\begin{bmatrix}I_k \\ Z\end{bmatrix}(I_k+Z^{\H}Z)^{-\frac12}=Q_*\wt{\Lambda}_*Q_*^{\H},\label{wls}\\
&(I_{n-k}+ZZ^{\H})^{-\frac12}[-Z, I_{n-k}] D\begin{bmatrix}-Z^{\H} \\ I_{n-k}\end{bmatrix}(I_{n-k}+ZZ^{\H})^{-\frac12}=Q_c\wt{\Lambda}_cQ_c^{\H},\label{wlc}\\
&[-Z, I_{n-k}] D\begin{bmatrix}I_k \\ Z\end{bmatrix}=0,\label{con3}
\end{align}
\end{subequations}
where
\begin{equation}\label{d}
D=[V_*, V_c]^{\H}\wt{A}(\wt{P}_*)[V_*, V_c].
\end{equation}
Then the perturbed NEPv \eqref{eq:nep2} has a solution $\wt{V}_*$ is equivalent to
\begin{itemize}
\item[(a)] there exists $Z$ such that \eqref{con3} holds;
\item[(b)] $\lambda_1(\wt{\Lambda}_c) - \lambda_{k}(\wt{\Lambda}_*)>0$.
\end{itemize}
Next, we first prove (a)  then (b).

\medskip

\noindent{\em Proof of $(a)$}
It follows from  \eqref{vav}, \eqref{con3} and \eqref{d} that
\begin{align*}
0&=[-Z, I_{n-k}] [V_*, V_c]^{\H}\wt{A}(\wt{P}_*)[V_*, V_c]\begin{bmatrix}I_k \\ Z\end{bmatrix}\\
&=\Lambda_c Z - Z \Lambda_* +(-ZV_*^{\H}+V_c^{\H})[\wt{A}(\wt{P}_*)-A(P_*)](V_*+V_cZ)\\
&= \mathbf{L}(Z) + \Phi(Z),
\end{align*}
where
\begin{align}
\mathbf{L}(Z)&=\Lambda_c Z - Z \Lambda_*,\notag\\
\Phi(Z)&=  (-ZV_*^{\H}+V_c^{\H})[\wt{A}(\wt{P}_*)-A(P_*)](V_*+V_cZ).\label{phiz}
\end{align}

Note that since $g$ defined in \eqref{g} is positive,  $\mathbf{L}(\cdot)$ is an invertible linear operator with
\begin{equation}\label{linv}
\|\mathbf{L}^{-1}\|_2^{-1}
=\min_{\lambda\in\lambda(\Lambda_*),\tilde{\lambda}\in\lambda(\Lambda_c)}|\lambda-\tilde{\lambda}|
=\lambda_{k+1}(A(P_*))-\lambda_k(A(P_*))=g >0.
\end{equation}
Therefore, we may define a mapping
$\mu: \mathbb{C}^{(n-k)\times k}\rightarrow \mathbb{C}^{(n-k)\times k}$ as
\begin{align}\label{muz}
\mu(Z)\equiv -\mathbf{L}^{-1}(\Phi(Z)).
\end{align}
%
%
By \eqref{vvz}, we have
\begin{align}
\|\wt{P}_* - P_* \|_2&= \|\wt{V}_*\wt{V}_*^{\H} - {V}_*{V}_*^{\rm H}\|_2\notag\\
& = \left\|[V_*, V_c]\begin{bmatrix} I_k \\ Z\end{bmatrix} (I_k+Z^{\H}Z)^{-1} [I_k, Z^{\H}] [V_*, V_c]^{\H} - V_*V_*^{\H}\right\|_2\notag\\
&= \left\|\begin{bmatrix} (I_k+Z^{\H}Z)^{-1}-I_k & (I_k+Z^{\H}Z)^{-1} Z^{\H}\\ Z(I_k+Z^{\H}Z)^{-1} & Z(I_k+Z^{\H}Z)^{-1} Z^{\H}\end{bmatrix}\right\|_2\notag\\
&= \frac{\|Z\|_2}{\sqrt{1+\|Z\|_2^2}}. \label{pp}
\end{align}
Then it follows from \eqref{phiz}, \eqref{da} and \eqref{pp} that
\begin{align}
\|\mathbf{L}^{-1}\Phi(Z)\|_2 &\le \frac{1}{g} (1+\|Z\|_2^2)\big(\delta + d  \|\wt{P}_*-P_*\|_2 \big)\notag\\
&=\frac{1}{g}\big( (1+\|Z\|_2^2)\delta  + d  \|Z\|_2\sqrt{1+\|Z\|_2^2} \big).\label{lphi}
\end{align}
Denote
\[
\mathbb{B}_{\eta_*}=\{Z\; | \; \|Z\|_2\le \eta_*\}.
\]
Note that $\mathbb{B}_{\eta_*}$ is a nonempty bounded closed convex set,
$\mu(\cdot)$ defined in \eqref{muz} is a continuous mapping,
and for any $Z\in\mathbb{B}_{\eta_*}$, by \eqref{lphi} and \eqref{feta}, it holds
\begin{align*}
\|\mu(Z)\|_2\le \frac{1}{g}\big( (1+\eta_*^2)\delta  + d \eta_*\sqrt{1+\eta_*^2} \big)=\eta_*,
\end{align*}
i.e., $\mu(Z)$ maps $\mathbb{B}_{\eta_*}$ into itself.
So by Brouwer's fixed-point theorem \cite{kham:14},
$\mu(Z)=Z$ has a fixed point $Z_*$ in $\mathbb{B}_{\eta_*}$.
In other words, \eqref{con3} has a solution $Z_*\in\mathbb{B}_{\eta_*}$.
This completes the proof of (a).

\medskip

\noindent{\em Proof of $(b)$}
If
\begin{align}\label{wll}
\min_{Q_*^{\H}Q_*=I_k}\|Q_*\wt{\Lambda}_*Q_*^{\H} - \Lambda_*\|_2 +
\min_{Q_c^{\H}Q_c=I_{n-k}}\|Q_c\wt{\Lambda}_cQ_c^{\H}-\Lambda_c\|_2<g,
\end{align}
then by Weyl Theorem \cite{stsu:1990}, we have
$|\lambda_k(\wt{\Lambda}_*)-\lambda_k(\Lambda_*)|
+|\lambda_1(\wt{\Lambda}_c)-\lambda_1(\Lambda_c)|< g$.
Consequently,
\begin{align*}
\lambda_1(\wt{\Lambda}_c)-\lambda_k(\wt{\Lambda}_*)
= g+[\lambda_1(\wt{\Lambda}_c)-\lambda_1({\Lambda}_c)] - [\lambda_k(\wt{\Lambda}_*)-\lambda_k({\Lambda}_*)] > g -g=0.
\end{align*}
Therefore, we only need to show \eqref{wll}, under the assumption $Z\in\mathbb{B}_{\eta_*}$.

We get by \eqref{da},  \eqref{tau}, and \eqref{d} that
\begin{align}
D&=[V_*, V_c]^{\H}{A}({P}_*)[V_*, V_c]+[V_*, V_c]^{\H}[\wt{A}(\wt{P}_*)-A(P_*)][V_*, V_c]\notag\\
&=\begin{bmatrix}\Lambda_* & 0\\ 0 & \Lambda_c\end{bmatrix}
+\Delta D,\label{ddd}
\end{align}
where $\Delta D = [V_*, V_c]^{\H}[\wt{A}(\wt{P}_*)-A(P_*)][V_*, V_c]$ satisfies
\begin{equation}\label{deld}
\|\Delta D\|_2= \|\wt{A}(\wt{P}_*)-A(P_*)\|\le \delta + d\|\wt{P}_*-P_*\|_2
\le \delta+d\tau_*.
\end{equation}

Let the singular value decomposition (SVD) of $Z$ be $Z=U_Z \Sigma_Z V_Z^{\H}$,
where $U_Z\in\mathbb{C}^{(n-k)\times k}$ has orthonormal columns,
$\Sigma_Z=\begin{bmatrix}\widehat{\Sigma}\\0\end{bmatrix}$,
$\widehat{\Sigma}=\diag(\sigma_1,\dots,\sigma_k)$, $\sigma_1\ge \dots\ge \sigma_k\ge 0$,
and $V_Z\in\mathbb{C}^{k\times k}$ is unitary.
Let $\sigma_i=\tan\theta_i$ for $i=1,\dots, k$, $\widehat{C}=\diag(\cos\theta_1,\dots,\cos\theta_k)$,
$\widehat{S}=\diag(\sin\theta_1,\dots,\sin\theta_k)$.
Then using \eqref{wls}, \eqref{ddd}, \eqref{deld}, we have
\begin{align}
\min_{Q_*^{\H}Q_*=I_k}\|Q_*\wt{\Lambda}_*Q_*^{\H}-\Lambda_*\|_2
&=\min_{Q_*^{\H}Q_*=I_k}\left\|Q_*V_Z[\widehat{C},\widehat{S}, 0] D \begin{bmatrix}\widehat{C}\\ \widehat{S}\\ 0\end{bmatrix} V_Z^{\H}Q_*^{\H}-\Lambda_*  \right\|_2\notag\\
&\le \left\|[\widehat{C},\widehat{S}, 0] D \begin{bmatrix}\widehat{C}\\ \widehat{S}\\ 0\end{bmatrix} -\Lambda_* \right\|_2\notag\\
&\le \|\Delta D\|_2 + \left\|\widehat{C}\Lambda_*\widehat{C}+[\widehat{S},0]\Lambda_c\begin{bmatrix}\widehat{S}\\ 0\end{bmatrix} - \Lambda_*\right\|_2\notag\\
&\le \delta + d\tau_* +h\sin^2\theta_1\notag\\
&\le \delta + d\tau_* + h\tau_*^2.\label{ll1}
\end{align}
Similarly,
\begin{align}
\min_{Q_c^{\H}Q_c=I_{n-k}}\|Q_c\wt{\Lambda}_cQ_c^{\H}-\Lambda_c\|_2
\le \delta + d\tau_* + h\tau_*^2.\label{ll2}
\end{align}

Direct calculations give rise to
\begin{subequations}\label{sll}
\begin{align}
2[\delta + d\tau_* + h\tau_*^2]-g
&= 2\left(\delta+ d \frac{\eta_{*}}{\sqrt{1+\eta_{*}^2}}\right) + 2h \frac{\eta_*^2}{1+\eta_*^2}-g\notag\\
&=2g\frac{\eta_{*}}{1+\eta_{*}^2} +2h\frac{\eta_*^2}{1+\eta_*^2}-g \label{g1} \\
&<2g\frac{\zeta}{1+\zeta^2} +2h\frac{\zeta^2}{1+\zeta^2}-g\label{g2}\\
&=\frac{2h\zeta^2-g(1-\zeta)^2}{1+\zeta^2}\notag\\
& =0,\label{g3}
\end{align}
\end{subequations}
where \eqref{g1} uses the fact $\eta_*$ is a root of \eqref{feta},
\eqref{g2} uses $\eta_*<\zeta$,
\eqref{g3} uses \eqref{gzeta}.
Combining \eqref{ll1}, \eqref{ll2} and \eqref{sll}, we get $(b)$.
This completes the proof.
\end{proof}

Note that $g>d$ is a necessary condition for that $f(\eta)=0$ has positive roots. Otherwise, $f(\eta)$ is always negative, and hence, $f(\eta)=0$ has no roots.
Next, we have several remarks in order.

\begin{remark}{\rm~
When the perturbation is sufficiently small, i.e., $\delta\ll 1$,
we have the following two claims: \\
(1) The assumption of Theorem~\ref{thm2} is weaker than that of Theorem~\ref{thm1}.\\
(2) The perturbation bound of Theorem~\ref{thm2} is shaper than that of Theorem~\ref{thm1}.\\
Claim (1) can be verified as follows.
Let the perturbation $\delta$ is sufficiently small and less than $\frac{1}{2}(g-d)\zeta$, we have
\begin{equation}\label{f2d}
f(\frac{2\delta}{g-d})=\frac{2g\delta}{g-d}- \frac{2d\delta}{g-d} -\delta + \mathcal{O}(\delta^2)=\delta+\mathcal{O}(\delta^2)>0.
\end{equation}
Note that $f(0)=-\delta<0$.
Therefore, $f(\eta)=0$ has at least one positive root within interval $(0,\frac{2\delta}{g-d})\subset (0, \zeta)$.
In other words, the assumption of Theorem~\ref{thm2}, which requires $f(\eta)=0$ has a positive root within $(0, \zeta)$,
is satisfied if $g>d$, provided that the perturbation is sufficiently small.
For the assumption of Theorem~\ref{thm1}, no matter how small the perturbation $\delta$ is, it requires $g>2d$.
Claim (2) can be verified as follows.
Using the  second order Taylor's expansion of $\sqrt{1+x}=1+\frac{1}{2}x-\frac{1}{8}x^2+\mathcal{O}(x^3)$,  we have by calculations,
\begin{align*}
f(\frac{\xi_*}{\sqrt{1-\xi_*^2}})
=\frac{g\delta^2}{(g-d)^2}+\mathcal{O}(\delta^3).
\end{align*}
Thus, $f(\frac{\xi_*}{\sqrt{1-\xi_*^2}})>0$ since $\delta\ll 1$.
Also note that $f(0)<0$, we know $\eta_*<\frac{\xi_*}{\sqrt{1-\xi_*^2}}$,
which leads to $\frac{\eta_*}{\sqrt{1+\eta_*^2}}<\xi_*$.
}\end{remark}

\begin{remark}{\rm
Note that $h>g$, then $\zeta$ defined in Theorem~\ref{thm2} is less than $\frac{1}{1+\sqrt{2}}$,
and $\tau_*$ is less than $\frac{1}{\sqrt{1+(1+\sqrt{2})^2}}\approx 0.3827$.
Therefore, when $\delta$ is not sufficiently small, Theorem~\ref{thm2} may not be applicable
since $\|\sin\Theta(\mathcal{R}(V_*),\mathcal{R}(\wt{V}_*))\|_2$ can be larger than 0.3827,
meanwhile Theorem~\ref{thm1} can be still applicable as long as $g>2d$.
}\end{remark}

\begin{remark}{\rm
Consider the following perturbation problem of a Hermitian matrix:
Given a Hermitian matrix $A_0$, a perturbation matrix $\Delta A_0$, which is also Hermitian.
Let  the eigenvalues of $A_0$ be $\lambda_1\le \dots \le \lambda_n$,
the column vectors of $V_*$ and $\wt{V}_*$ be  the eigenvectors of $A_0$ and $A_0+\Delta A_0$ associated with their $k$ smallest eigenvalues, respectively.
Assume $g=\lambda_{k+1}-\lambda_k>0$.
What's the upper bound for $\|\sin\Theta(\mathcal{R}(V_*),\mathcal{R}(\wt{V}_*))\|_2$?

Note that since $d=0$, \eqref{feta} becomes a quadratic equation of $\eta$.
It is easy to see that it has positive roots if and only if $g\ge 2\delta$.
And when $g\ge 2\delta$, it has two positive roots, and the smaller one is $\frac{2\delta}{g+\sqrt{g^2-4\delta^2}}$.
Then Theorem~\ref{thm2} can be rewritten as:
\begin{center}
{\em \mbox{\rm If} $\delta\le \frac{1}{2}g$ \mbox{\rm and} $\frac{2\delta}{g+\sqrt{g^2-4\delta^2}}<\zeta$,
{\rm then} $\|\tan\Theta(\mathcal{R}(V_*),\mathcal{R}(\wt{V}_*))\|_2\le \frac{2\delta}{g+\sqrt{g^2-4\delta^2}}$.}
\end{center}
This conclusion is similar to the perturbation theorems in \cite[Chapter V, subsection 2.2]{stsu:1990}.
}\end{remark}

\subsection{Condition number}
In this subsection, we provide  a condition number for NEPv \eqref{eq:nep}.
Recall the theory of condition developed by Rice \cite{rice:1966}, also note that
\[
\frac{\|P_*-\wt{P}_*\|_2}{\|P_*\|_2}=\|\sin\Theta(\mathcal{R}(V_*),\mathcal{R}(\wt{V}_*))\|_2.
\]
We may define a condition number as
\begin{align}
\kappa=\lim_{\epsilon\to 0}\Big\{\frac{\|\sin\Theta(\mathcal{R}(V_*),\mathcal{R}(\wt{V}_*))\|_2}{\epsilon} \; \big| \; \delta\le \epsilon,
\mbox{$V_*$, $\wt{V}_*$  are the solutions to \eqref{eq:nep} and \eqref{eq:nep2}}, \\
\mbox{respectively, $\delta$ is defined in \eqref{ad}} \Big\}.\notag
\end{align}

Now using the second-order Taylor's expansion of $(1+x)^{1/2}$, by \eqref{xistar}, we have
\begin{align}\label{xistar_Taylor}
\xi_*= \frac{1}{g-d}\delta+O(\delta^2).
\end{align}
Combining it with Theorem~\ref{thm1}, we can obtain the first order absolute perturbation bound for the eigenvector subspace  $V_*$:
\begin{align}
\|\sin\Theta(\mathcal{R}(V_*),\mathcal{R}(\wt{V}_*))\|_2\le \frac{1}{g-d}\delta+O(\delta^2).
\end{align}
Then it follows
\[
\frac{\|\sin\Theta(\mathcal{R}(V_*),\mathcal{R}(\wt{V}_*))\|_2}{\epsilon}\lesssim \frac{1}{g-d}.
\]
Therefore, we may define a condition number for NEPv  \eqref{eq:nep} as
\begin{align}\label{condnum}
\kappa\equiv \frac{1}{g-d}.
\end{align}
This form can also be derived from Theorem~\ref{thm2}.
In fact, letting $\delta\rightarrow 0$ in \eqref{feta}, by \eqref{f2d}, we know that $\eta_*$ is less than $\frac{2\delta}{g-d}$, thus, $\eta_*\rightarrow 0$.
Then \eqref{feta} can be rewritten as
\[
g\eta-d\eta+\delta\approx 0.
\]
Therefore, $\eta_*\approx \frac{\delta}{g-d}$, and $\frac{\eta_*}{\sqrt{1+\eta_*^2}}\approx \frac{\delta}{g-d}$.
Thus, by Theorem~\ref{thm2}, we have
\[
\|\sin\Theta(\mathcal{R}(V_*),\mathcal{R}(\wt{V}_*))\|_2\lesssim \frac{1}{g-d}\delta,
\]
 from which we may define a condition number as in \eqref{condnum}.

\medskip

Recall that $g$ is the gap between the $k$th and $k+1$st smallest eigenvalues of $A(P_*)$,
and $d$ is a local Lipschitz constant for the inequality $\|A(P)-A(P_*)\|_2\le d \|P-P_*\|_2$.
Thus, the newly defined condition number $\kappa$,
which can be used to measure the sensitivity of NEPv at $V_*$,
depends on the eigenvalue gap as well as the sensitivity of $A(P)$ at $P=P_*$.
A large $g$ and a small $d$ will ensure a good conditioned  NEPv \eqref{eq:nep}.

\medskip

\begin{remark}\label{rem:pertbd}{\rm
Notice that $\delta$ can be used to measure the magnitude of the backward error (see \eqref{backerr} below).
Then using the  rule of thumb -- ``forward error $\lesssim$ backward error $\times$ condition number'',
we may use $\frac{\delta}{g-d}$ as an approximate perturbation bound.
}\end{remark}

\subsection{Error bounds} In this subsection we give two error bounds for NEPv \eqref{eq:nep},
which can be used to measure the quality of approximate solutions to NEPv \eqref{eq:nep}.

Let $\widehat{V}\in\mathbb{V}_k$ be an approximate solution to NEPv \eqref{eq:nep}, and denote the residual by
\begin{equation}\label{eq:r}
R=A(\widehat{P})\widehat{V} - \widehat{V} [ \widehat{V}^{\H}A(\widehat{P})\widehat{V}],
\end{equation}
where $\widehat{P}=\widehat{V}\widehat{V}^{\H}\in\mathbb{P}_k$.
It is easy to verify that \eqref{eq:r} can be rewritten as
\begin{align}\label{backerr}
\widehat{A}(\widehat{P})\widehat{V}=\widehat{V} [ \widehat{V}^{\H}\widehat{A}(\widehat{P})\widehat{V}],
\end{align}
where
\begin{align*}
\widehat{A}(\widehat{P}) &= A_0+\Delta A_0 + A_1(\widehat{P})+A_2(\widehat{P}),\\
\Delta A_0 &= - R\widehat{V}^{\H} - \widehat{V} R^{\H}.
\end{align*}
Now we take \eqref{eq:nep} as a perturbed NEPv of \eqref{backerr},
where only the constant matrix $A_0$ is perturbed, the matrix functions $A_1$ and $A_2$ remain unchanged.
Noticing that $\delta_0= \|R\widehat{V}^{\H} + \widehat{V} R^{\H}\|_2=\|R\|_2$, $\delta_1=\delta_2=0$ and $\delta= \|R\|_2$,
we can rewrite Theorems~\ref{thm1} and \ref{thm2}  as the following two corollaries.

\begin{corollary}\label{cor1}
Let $\widehat{V}$ be an approximate solution to NEPv \eqref{eq:nep}, $\widehat{P}=\widehat{V}\widehat{V}^{\H}$, $R$ be given by \eqref{eq:r}.
Define $\hat{d}$ as $d$ in \eqref{ad} by replacing $P_*$ by $\widehat{P}$, and
assume
\begin{equation}\label{hatg}
\hat{g}=\lambda_{k+1}(\widehat{A}(\widehat{P})) -\lambda_{k}(\widehat{A}(\widehat{P}))>0.
\end{equation}
If
\begin{equation}\label{rgd}
\|R\|_2< \frac12 \hat{g} - \hat{d},
\end{equation}
then NEPv \eqref{eq:nep} has a solution ${V}_*\in\mathbb{V}_{\hat{\xi}_*}$ with
\begin{align}\label{xistar2}
\hat{\xi}_*= \frac{2\|R\|_2}{\hat{g} -\hat{d} - \|R\|_2 + \sqrt{(\hat{g} - \hat{d} - \|R\|_2)^2-4\hat{d}\|R\|_2}}.
\end{align}
\end{corollary}

\begin{corollary}\label{cor2}
Let $\widehat{V}$ be an approximate solution to NEPv \eqref{eq:nep}, $\widehat{P}=\widehat{V}\widehat{V}^{\H}$, $R$ be given by \eqref{eq:r}.
Assume \eqref{hatg}, define $\hat{d}$ as in Corollary~\ref{cor1}, and denote
\begin{equation}
\hat{h}=\max_{1\le j\le k} [\lambda_{k+j}(\widehat{A}(\widehat{P}))  -  \lambda_j(\widehat{A}(\widehat{P}))],
\qquad
\hat{\zeta}=\frac{\sqrt{\hat{g}}}{\sqrt{\hat{g}}+\sqrt{2\hat{h}}}.
\end{equation}
Suppose that $\|R\|_2$ is sufficiently small such that
\begin{equation}\label{fhateta}
\hat{f}(\eta)\equiv \hat{g} \eta - \hat{d}\eta\sqrt{1+\eta^2} - (1+\eta^2)\|R\|_2 = 0
\end{equation}
has positive roots, and its smallest positive root, denoted by $\hat{\eta}_*$, is smaller than $\hat{\zeta}$.
Then the NEPv \eqref{eq:nep} has a solution ${V}_*\in\mathbb{V}_{\hat{\tau}_*}$
with
\begin{equation}\label{tau2}
\hat{\tau}_*=\frac{\hat{\eta}_*}{\sqrt{1+\hat{\eta}_*^2}}.
\end{equation}
\end{corollary}

It is worth mentioning here that both \eqref{xistar2} and \eqref{tau2} are computable as long as $\hat{g}$ and $\hat{d}$ are available.

\begin{remark}\label{rem:errbd}{\rm
By \eqref{backerr}, we can use $\delta=\|\Delta A_0\|_2=\|R\|_2$ to measure the magnitude of the backward error.
Recall the condition number $\kappa$ we defined in \eqref{condnum} and
the thumb rule, we may use $\frac{\|R\|_2}{\hat{g}-\hat{d}}$ as an approximate error bound,
where $\hat{g}$ is given by \eqref{hatg}.
}\end{remark}

\section{Applications}\label{sec:appl}
In this section,  we apply our theoretical results to two  practical problems:
the Kohn-Sham equation and the trace ratio optimization.
All numerical experiments are carried out 
using {\tt MATLAB} R2016b, with machine epsilon $\epsilon\approx 2.2\times 10^{-16}$.

The exact solution $V_*$ to NEPv \eqref{eq:nep}  is approximated by $\widehat{V}_*$, which is obtained by solving NEPv \eqref{eq:nep} via  SCF iteration
with stopping criterion
\[
\frac{\|A(\widehat{V}_*\widehat{V}_*^{\H})\widehat{V}_*-\widehat{V}_* [\widehat{V}_*^{\H} A(\widehat{V}_*\widehat{V}_*^{\H})\widehat{V}_*]\|_2}{\|A(\widehat{V}_*\widehat{V}_*^{\H})\|_2}\le 10^{-14}.
\]
And the exact solution $\wt{V}_*$ to NEPv \eqref{eq:nep2} is approximated similarly.
At the $l$th SCF iteration, an approximate solution $V_l$ is obtained.
Then we can use $V_l$ to validate our error bounds,
which will tell us how far away the approximate solution $V_l$ from the exact solution $V_*$.

The following notations will be used to illustrate our results.
The solution perturbation $\|\sin\Theta(\mathcal{R}(V_*),\mathcal{R}(\wt{V}_*))\|_2$, the perturbation bound given by Theorems~\ref{thm1} and \ref{thm2},
and Remark~\ref{rem:pertbd} are denoted by $\chi_*$, $\xi_*$, $\tau_*$ and $\gamma_*$, respectively.
For the approximate solution $V_l$, the solution error $\|\sin\Theta(\mathcal{R}(V_*),\mathcal{R}(V_l))\|_2$ and the error bounds given by Corollaries~\ref{cor1}, \ref{cor2} and Remark~\ref{rem:errbd}
are denoted by $\hat{\chi}_*$, $\hat{\xi}_*$, $\hat{\tau}_*$ and $\hat{\gamma}_*$, respectively.

\subsection{Application to the Kohn-Sham equation}\label{s:ks}
We consider the perturbation of the discretized KS equation:
\begin{align}\label{e:KS}
H(V)V=V\Lambda,
\end{align}
where $V\in\mathbb{R}^{n\times k}$ is orthonormal,
the discretized Hamiltonian $H(V)\in\mathbb{R}^{n\times n}$ is a matrix function with respect to $V$,
and $\Lambda\in\mathbb{R}^{k\times k}$ is a diagonal matrix consisting of $k$ smallest eigenvalues of $H(V)$.
In particular, we consider the discretized Hamiltonian in the  form of
\begin{equation}\label{e:HX}
H(V)= \frac{1}{2}L+V_{\rm ion}+{\rm Diag}(L^\dag \rho)-2\gamma{\rm Diag}(\rho^{\frac{1}{3}}),
\end{equation}
where $L$ is a finite dimensional representation of the Laplacian operator,
$V_{\rm ion}$ is the ionic pseudopotentials sampled on the suitably chosen Cartesian grid,
$L^\dag$ denotes the pseudoinverse of $L$,
$\rho={\rm diag}(VV^{\T})$ denotes the vector containing the diagonal elements of the matrix $VV^{\T}$,
and $\Diag(x)$ denotes a diagonal matrix with $x$ on its diagonal.
The last term of \eqref{e:HX} is derived from $e_{xc}(\rho)$ defined in \cite[equation (2.11)]{liww:15}.

Let
\begin{align*}
A_0= \frac{1}{2}L+V_{\rm ion}, \quad
A_1(P)= {\rm Diag}(L^\dag \rho(P)),\quad
A_2(P)=-2\gamma {\rm Diag}( \rho(P)^{\frac{1}{3}}),
\end{align*}
where $P= VV^{\T}$. Then the discretized Hamiltonian  $H(V)$ can be rewritten as
\[
A(P)=A_0+A_1(P)+A_2(P).
\]
Thus, the KS equation \eqref{e:KS} with $H(V)$ given by \eqref{e:HX} can be written in the form of \eqref{eq:nep} with \eqref{eq:a012}, indeed.

Next, we set the perturbed KS equation as in the form \eqref{eq:nep2} with
\begin{eqnarray*}
\wt{A}_0&:= &\frac{1}{2}L+V_{\rm ion}+\Delta L+\Delta V_{\rm ion},\\
\wt{A}_1(\wt{P}_*)&:= &{\rm Diag}((L+\Delta L)^\dag \rho(\wt{P}_*)),\\
\wt{A}_2(\wt{P}_*)&:= &-2\gamma {\rm Diag}( \rho(\wt{P}_*)^{\frac{1}{3}}).
\end{eqnarray*}
Then according to \eqref{ad}, we have
\begin{align*}
\delta_0  &= \| \Delta L+\Delta V_{\rm ion}\|_2,  \\
\delta_1  &=\sup_{P\in\mpxi} \|{\rm Diag}((L+\Delta L)^\dag - L^\dag)\rho(P)\|_2,\\
\delta_2  &= 0,\\
d_1  &=\sup_{P\ne P_*,P\in\mpxi} \frac{\|{\rm Diag}((L+\Delta L)^\dag \rho(P)-L^\dag \rho(P_*))\|_2}{\|P-P_*\|_2},\\
d_2  &= 2\gamma \sup_{P\ne P_*,P\in\mpxi} \frac{\|{\rm Diag}( \rho(P)^{\frac{1}{3}}- \rho(P_*)^{\frac{1}{3}})\|_2}{\|P-P_*\|_2}.
\end{align*}
In our numerical tests, $L$, $V_{\rm ion}$, $\Delta L$ and $\Delta V_{\rm ion}$  are generated
by using the {\tt MATLAB} built-in functions {\tt eye}, {\tt diag}, {\tt ones}, {\tt zeros}, and {\tt sprandsym} as follows:
\begin{eqnarray*}
&& L=\texttt{eye}(n)-\texttt{diag}(\texttt{ones}(n-1,1),1); \quad  L=(L+L')/h^2; \\
&& V_{\rm ion}= {\tt zeros}(n);\\
&& \Delta L = \epsilon_1*L; \\
&& \Delta V_{\rm ion}=\epsilon_2*\texttt{sprandsym}(n,0.5).
\end{eqnarray*}
Here $n$ is the matrix size, $h$ denotes the step size,  $\epsilon_1$,$\epsilon_2$ are two parameters used to control the magnitude of the perturbation.

%


Set $n=50$, $k=8$, $\epsilon_1=\epsilon_2=\epsilon=10^{-j}$ with $j=3,4,\dots, 12$.
In Figure \ref{ks-figh5678},  we plot $\chi_*$, $\xi_*$ and $\tau_*$ versus $\epsilon$ for four different step sizes $h=0.05,0.06,0.07,0.08$.
In Table \ref{ks-tableh5678}, we lists  $\frac{g}{d}$, $\frac{1}{g-d}$, $\chi_*$, $\xi_*$, $\tau_*$, and $\gamma_*$ for different $\epsilon$.
We can observe that the perturbation bounds $\xi_*$, $\tau_*$ and $\gamma_*$ are good upper bounds for the solution perturbation $\chi_*$,  while $\tau_*$ is sharper,
especially when $\frac{g}{d}$ is close to one.
And as $h$ increases, $\frac{g}{d}$ decreases, the condition number $\frac{1}{g-d}$ increases, and as a result, the perturbation bounds become less sharp.
Also note that, when $h=0.08$, the assumption of Theorem~\ref{thm1}  does not hold since $\frac{g}{d}<2$,
thus, $\xi_*$ is no longer available (denoted by ``-'' in Table~\ref{ks-tableh5678}) and
we can only use Theorem~\ref{thm2} in this case.

\begin{figure}[ht]
\centering
\includegraphics[width=0.49\textwidth]{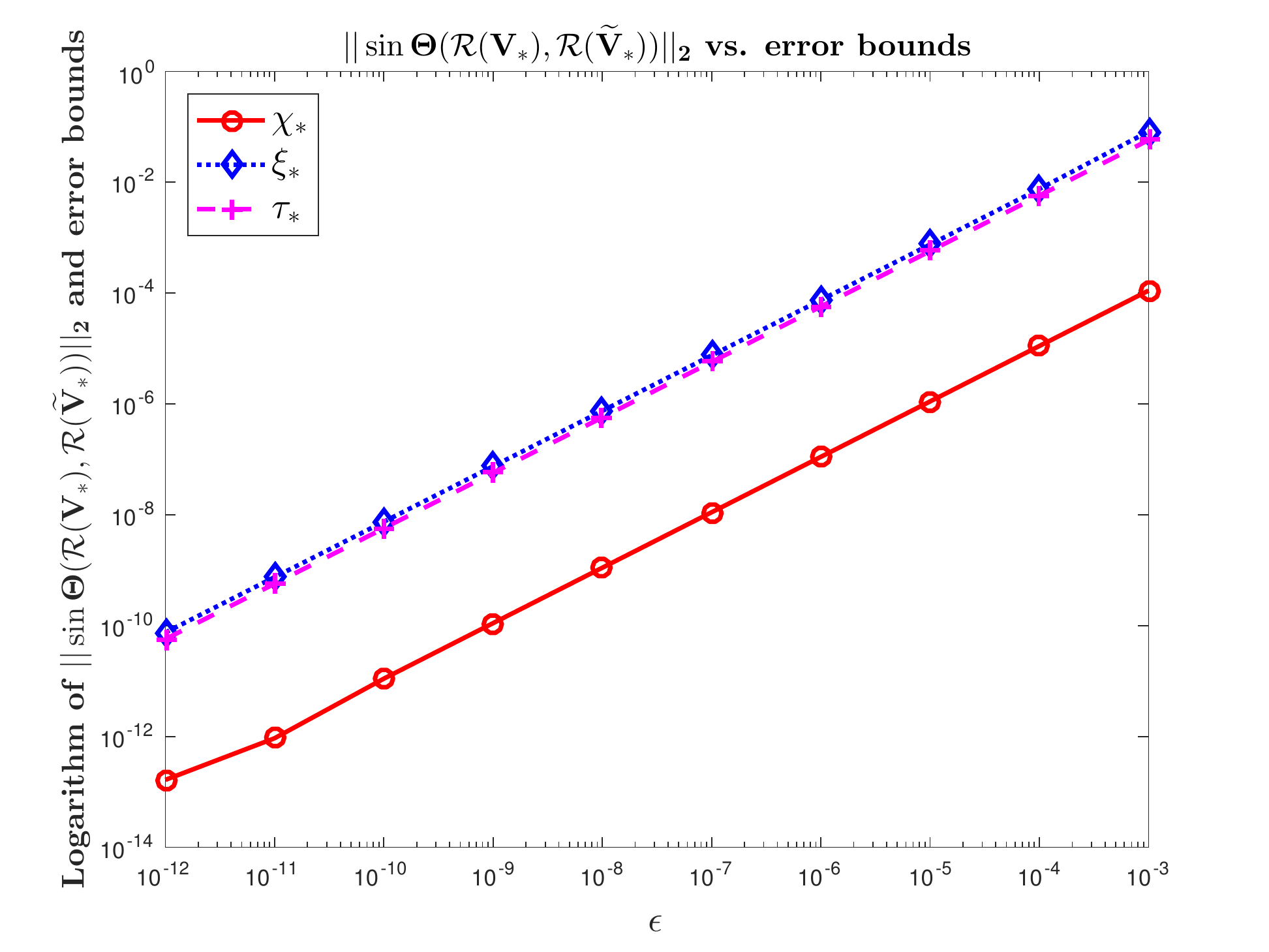}
\includegraphics[width=0.49\textwidth]{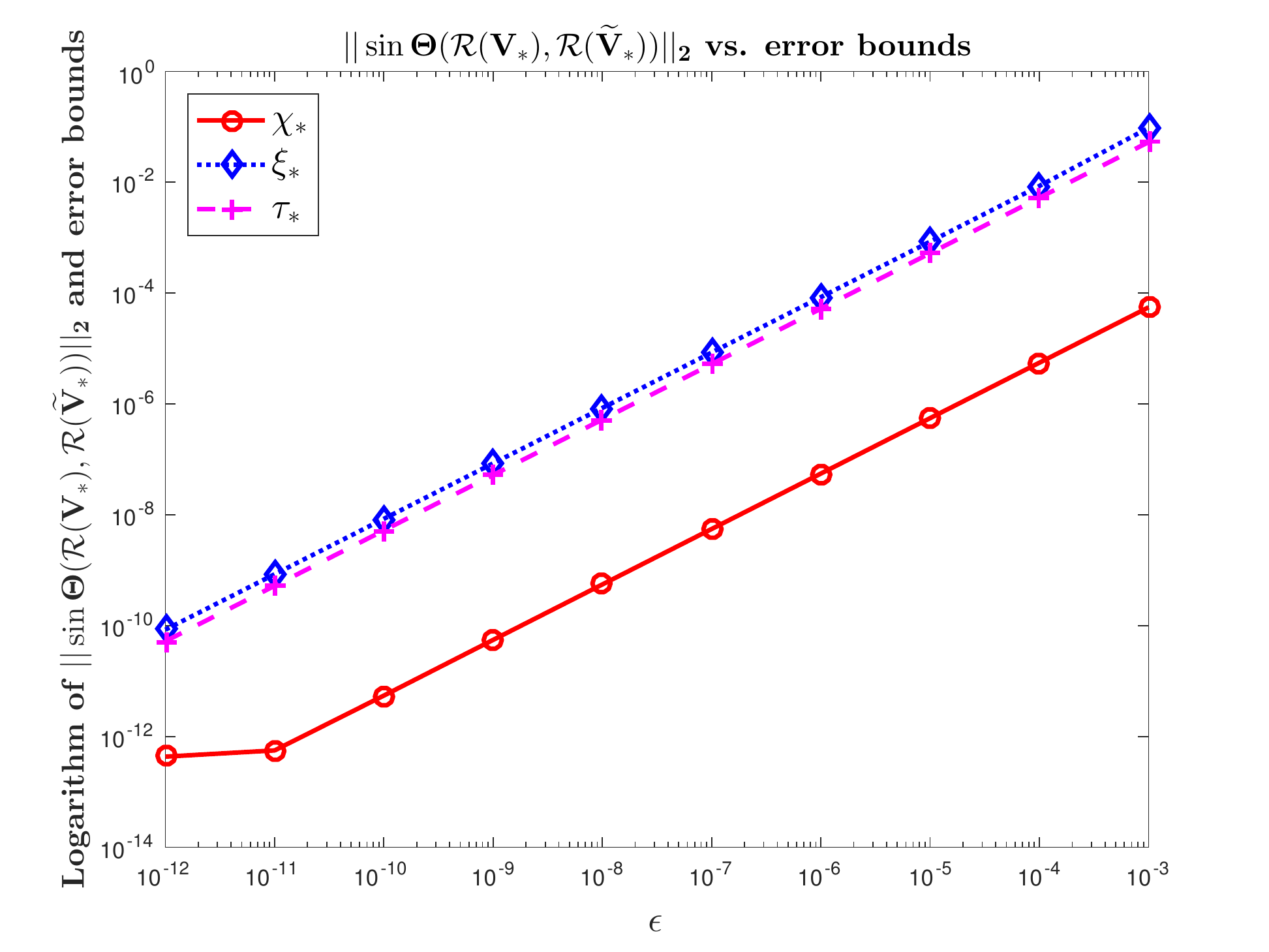}\\
        \small{$h=0.05$}\hspace{2.1in}
        \small{$h=0.06$}\\
\includegraphics[width=0.49\textwidth]{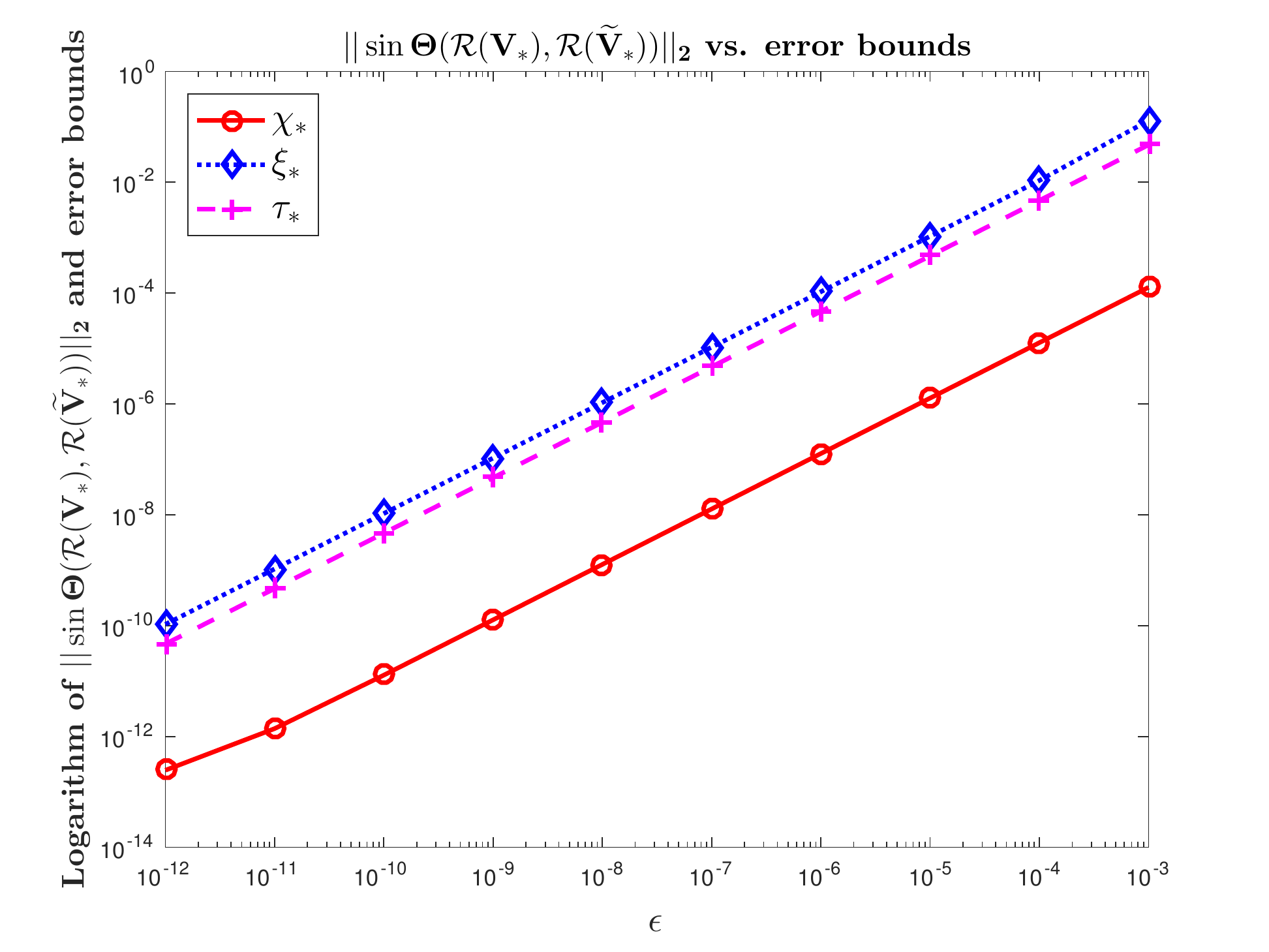}
\includegraphics[width=0.49\textwidth]{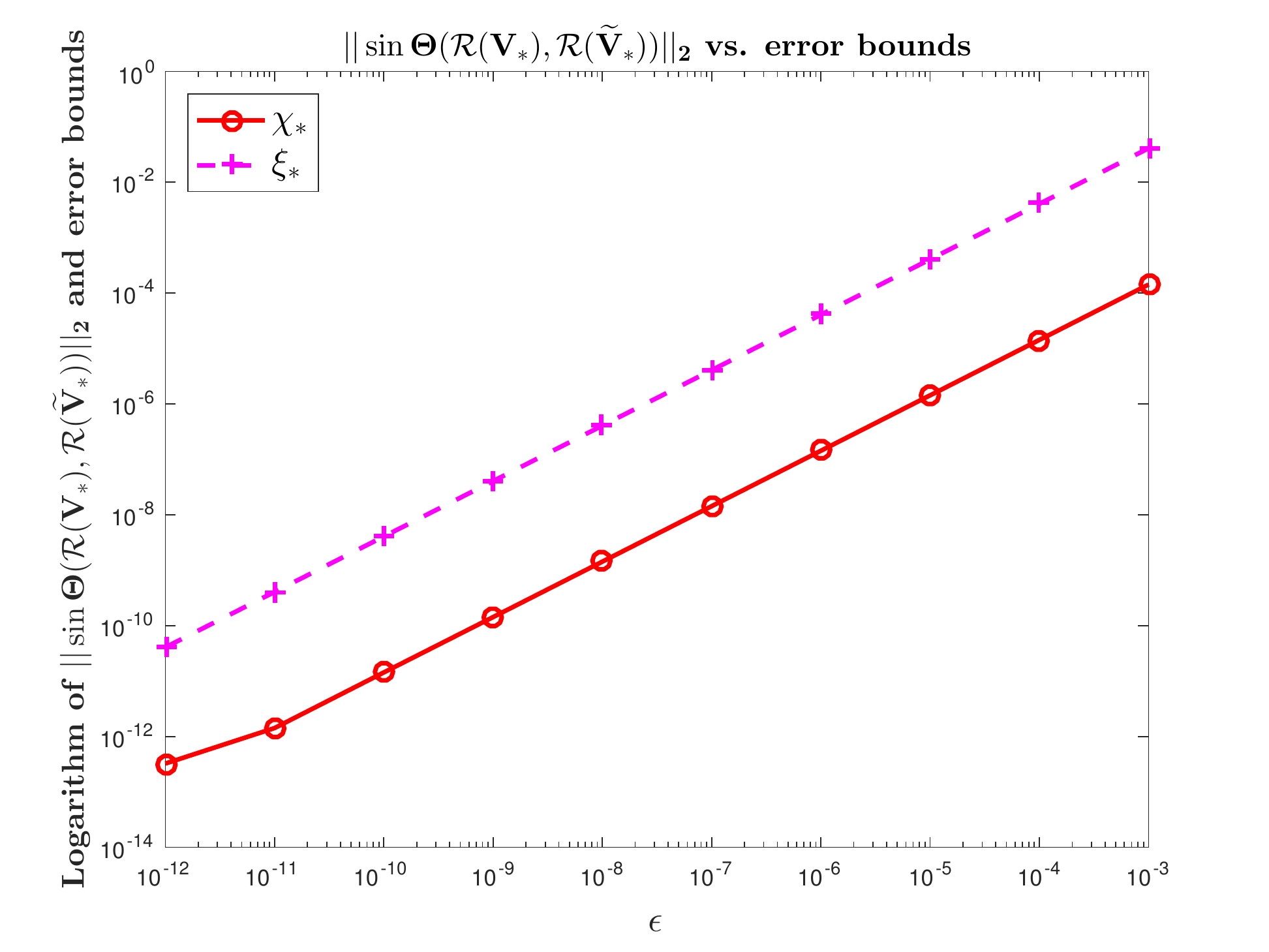}\\
        \small{$h=0.07$}\hspace{2.1in}
        \small{$h=0.08$}\\
\caption{$\|\sin\Theta(\mathcal{R}(V_*),\mathcal{R}(\wt{V}_*))\|_2$ vs. perturbation bounds  for the KS equation 
}
 \label{ks-figh5678}
\end{figure}

\begin{table}\renewcommand{\arraystretch}{1.2} \addtolength{\tabcolsep}{2pt}
  \caption{Perturbation bounds for the KS equation 
  }\label{ks-tableh5678}
  \begin{center} {\scriptsize
   \begin{tabular}[c]{|c|c|c|c|c|c|c|}
    \hline
\multicolumn{7}{|c|}{$h=0.05$} \\  \hline
$\epsilon$& $g/d$ & $1/(g-d)$  & $\chi_*$ & $\xi_*$ & $\tau_*$  & $\gamma_*$ \\  \hline
$10^{-12}$  &  7.4120e+00&   9.3503e-02&   1.6497e-13 &  7.4924e-11&   5.7110e-11&   7.4924e-11   \\
$ 10^{-10}$  & 7.4120e+00 &  9.3503e-02 &  1.1055e-11 &  7.4925e-09&   5.7111e-09&   7.4925e-09  \\
$ 10^{-8}$  &  7.4120e+00&   9.3503e-02 &  1.1093e-09&   7.4925e-07&   5.7111e-07&   7.4925e-07 \\
$ 10^{-6}$  &7.4120e+00  & 9.3503e-02 &  1.1093e-07 &  7.4931e-05&   5.7113e-05  & 7.4925e-05 \\
$ 10^{-4}$   &7.4120e+00  & 9.3503e-02 &  1.1092e-05&   7.5580e-03&   5.7295e-03  & 7.4925e-03  \\
 \hline
\multicolumn{7}{|c|}{$h=0.06$} \\  \hline
$\epsilon$& $g/d$ & $1/(g-d)$  & $\chi_*$ & $\xi_*$ & $\tau_*$  & $\gamma_*$ \\  \hline
$10^{-12}$ & 4.2452e+00 & 1.5213e-01 & 1.6599e-13 & 8.4634e-11 & 5.2363e-11 & 8.4634e-11\\
$ 10^{-10}$& 4.2452e+00 & 1.5213e-01 & 1.2266e-11 & 8.4635e-09 & 5.2363e-09 & 8.4635e-09\\
$ 10^{-8}$ & 4.2452e+00 & 1.5213e-01 & 1.2289e-09 & 8.4635e-07 & 5.2363e-07 & 8.4635e-07\\
$ 10^{-6}$ & 4.2452e+00 & 1.5213e-01 & 1.2289e-07 & 8.4644e-05 & 5.2365e-05 & 8.4635e-05\\
$ 10^{-4}$ & 4.2452e+00 & 1.5213e-01 & 1.2288e-05 & 8.5585e-03 & 5.2493e-03 & 8.4635e-03\\
 \hline
\multicolumn{7}{|c|}{$h=0.07$} \\  \hline
$\epsilon$ & $g/d$ & $1/(g-d)$  & $\chi_*$ & $\xi_*$ & $\tau_*$  & $\gamma_*$ \\  \hline
$10^{-12}$&   2.5866e+00&   2.5755e-01&   2.4601e-13&   1.0550e-10&   4.6671e-11&   1.0550e-10 \\
$ 10^{-10}$&   2.5866e+00&   2.5755e-01&   1.2805e-11&   1.0550e-08&   4.6670e-09&   1.0550e-08    \\
$ 10^{-8}$ &   2.5866e+00&   2.5755e-01&   1.2717e-09&   1.0550e-06&   4.6670e-07&   1.0550e-06     \\
$ 10^{-6}$ &   2.5866e+00&   2.5755e-01&   1.2717e-07&   1.0552e-04&   4.6671e-05&   1.0550e-04 \\
$ 10^{-4}$ &   2.5866e+00&   2.5755e-01&   1.2716e-05&   1.0736e-02&   4.6756e-03&   1.0550e-02\\
 \hline
\multicolumn{7}{|c|}{$h=0.08$} \\  \hline
$\epsilon$& $g/d$ & $1/(g-d)$  & $\chi_*$ & $\xi_*$ & $\tau_*$  & $\gamma_*$ \\  \hline
$10^{-12}$ &  1.6602e+00&   5.1773e-01&   1.4211e-12&   -&   4.0590e-10&   1.6355e-09 \\
$ 10^{-10}$ &  1.6602e+00&   5.1773e-01&   1.4318e-11&  -&   4.0590e-09&   1.6355e-08    \\
$ 10^{-8}$   & 1.6602e+00 &  5.1773e-01 &  1.4276e-09 & - &  4.0590e-07 &  1.6355e-06    \\
$ 10^{-6}$  &1.6602e+00  & 5.1773e-01  & 1.4276e-07  & -  & 4.0590e-05  & 1.6355e-04  \\
$ 10^{-4}$   &1.6602e+00   &5.1773e-01   &1.4275e-05   &-  &4.0645e-03   &1.6355e-02\\
 \cline{2-7}\hline
\end{tabular} }
  \end{center}
\end{table}

\begin{figure}[ht]
\centering
\includegraphics[width=0.75\textwidth]{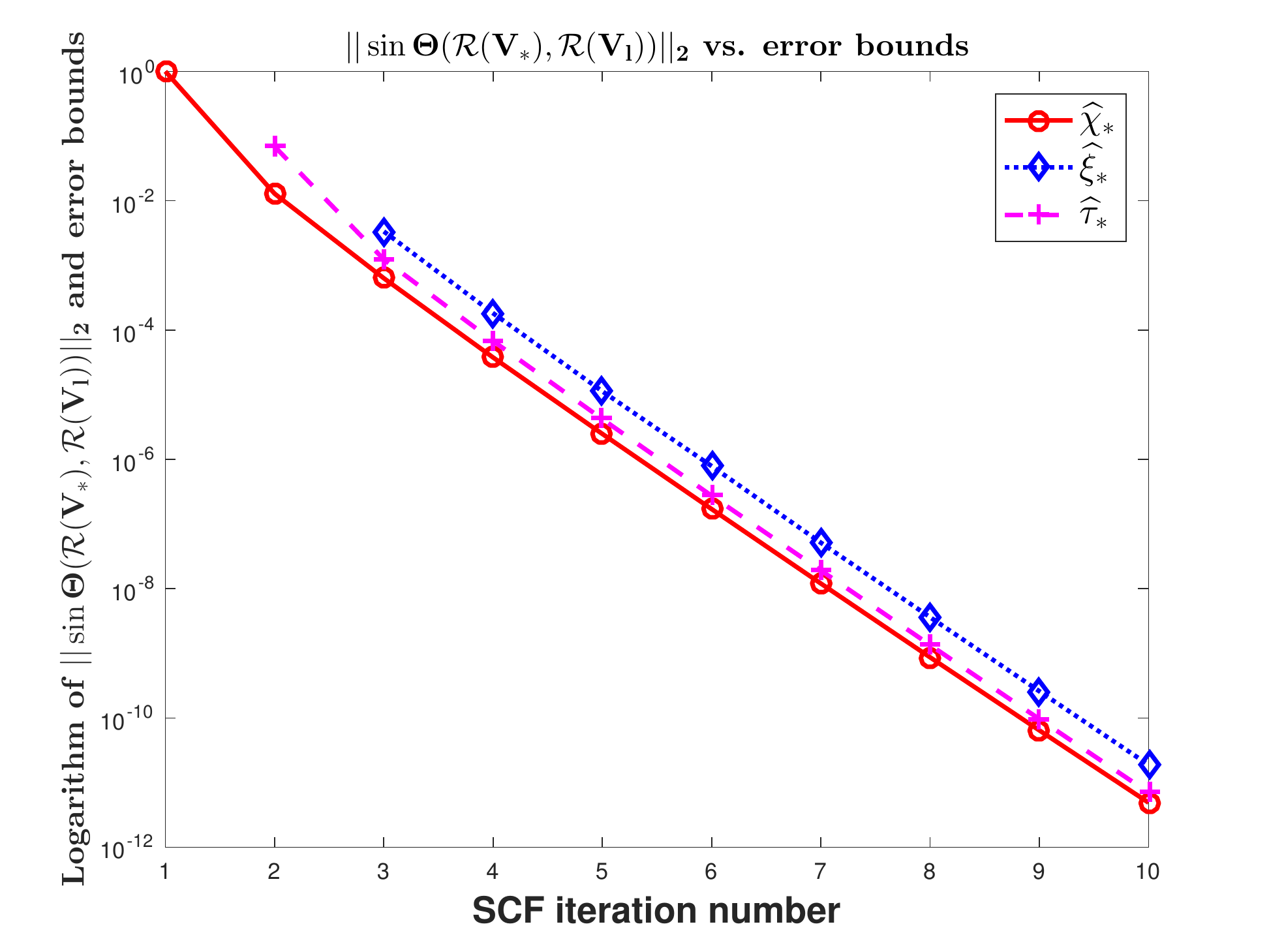}
\caption{$\|\sin\Theta(\mathcal{R}(V_*),\mathcal{R}(V_l))\|_2$ vs. error bounds for the  KS equation}
 \label{ks-fig2}
\end{figure}

Set $n=50$, $k=4$, $h=0.04$.
Figure \ref{ks-fig2} displays $\hat{\chi}_*$, the error bounds $\hat{\xi}_*$ and $\hat{\tau}_*$.
We can see from Figure~\ref{ks-fig2} that as SCF iterations converge, $\hat{\chi}_*$, $\hat{\xi}_*$ and $\hat{\tau}_*$ decrease linearly.
The error bounds $\hat{\xi}_*$ and $\hat{\tau}_*$  are good upper bounds for $\hat{\chi}_*$, and the latter one is sharper.
Also note that $\hat{\tau}_*$ is applicable from the second iteration, meanwhile $\hat{\xi}_*$ is applicable from the third,
which indicates that Corollary~\ref{cor2} has weaker assumption than that of Corollary~\ref{cor1} in this case.

\subsection{Application to the trace ratio optimization}
We consider the following maximization problem of the sum of the trace ratio:
\begin{equation}\label{e:TR}
\max_{V\in\mathbb{R}^{n\times k}, V^{\T}V=I_k} f(V):=\frac{{\rm tr}(V^{\T}AV)}{{\rm tr}(V^{\T}BV)}+{\rm tr}(V^{\T}CV),
\end{equation}
where ${\rm tr}(\cdot)$ means the trace of a square matrix, $A,B,C\in\mathbb{R}^{n\times n}$ are real symmetric with $B$ being positive definite, and $k<n$.

 As shown in \cite{zhli:2014a},  any critical point $V$ of (\ref{e:TR}) is a solution to the following nonlinear eigenvalue problem
 \begin{equation}\label{e:tr-kkt}
E(V)V=V(V^{\T}E(V)V),
\end{equation}
where
\[
E(V)=A\frac{1}{\phi_B(V)}-B\frac{\phi_A(V)}{\phi^2_B(V)} +C,
\]
and for any symmetric matrix $S$, $\phi_S(V)$ is defined as $\phi_S(V):={\rm tr}(V^{\T}SV)$.
Moreover, if $V$ is a global maximizer, then it is an orthonormal eigenbasis of $E(V)$ corresponding to its $k$ largest eigenvalues.

Let $P=VV^{\T}$, and note that $\phi_A(V)={\rm tr}(AP)$ and $\phi_B(V)={\rm tr}(BP)$ are functions of $P$,
then by setting
\[
A_0=C, \quad A_2(P)=A\frac{1}{\phi_B(V)}-B\frac{\phi_A(V)}{\phi^2_B(V)},
\]
the Problem (\ref{e:tr-kkt}) can be rewritten as
\begin{equation}\label{e:tr-kkt2}
(A_0+A_2(P))V=V(V^{\T}(A_0+A_2(P))V),
\end{equation}
which is of the form \eqref{eq:nep} with $A_1(P)\equiv 0$.

Suppose that $A,B,C$ are perturbed slightly, we have the following perturbed equation of  (\ref{e:tr-kkt2}):
\begin{equation}\label{e:tr-kkt2-pert}
(\widetilde{A}_0+\widetilde{A}_2(\widetilde{P}))\widetilde{V}=\widetilde{V}(\widetilde{V}^{\T}(\widetilde{A}_0+\widetilde{A}_2(\widetilde{P}))\widetilde{V}),
\end{equation}
where
\begin{eqnarray*}
&&\widetilde{P}=\widetilde{V}\widetilde{V}^{\T},\quad \widetilde{A}_0=A_0+\Delta C=C+\Delta C,\\
&&\widetilde{A}_2(\widetilde{P})=(A+\Delta A)\frac{1}{\phi_{B+\Delta B}(\widetilde{V})} -(B+\Delta B)\frac{\phi_{A+\Delta A}(\widetilde{V})}
{\phi^2_{B+\Delta B}(\widetilde{V})},
\end{eqnarray*}
and $\Delta A$, $\Delta B$, $\Delta C$ are real symmetric matrices.

Then by calculations, we have
\begin{eqnarray*}
&&\delta_0  = \| \Delta C\|_2,  \\
&&\delta_2  = \sup_{P\in\mpk} {\|\wt{A}_2(P) - A_2(P)\|_2}\le \|A\|_2\frac{\Omega_{\Delta B}}{\omega_{B+\Delta B} \omega_B}\\
&&\qquad +\|B\|_2\frac{\Omega_{\Delta A}\Omega_B^2+\Omega_A(\Omega_B+\Omega_{B+\Delta B})\Omega_{\Delta B}}{\omega_{B+\Delta B}^2  \omega_B^2} + \|\Delta A\|_2\frac{1}{\omega_{B+\Delta B} } + \|\Delta B\|_2\frac{\Omega_{A+\Delta A}}{\omega_{B+\Delta B}^2 },\\
&& d=d_2 = \sup_{{P}\ne P_*, {P}\in \mpk}  \frac{\|A_2({P}) - A_2(P_*)\|_2}{\|{P} - P_*\|_2}\le\frac{2\|A\|_2\|B\|_2}{\omega_B^2}
+ \frac{2 \|B\|_2^2\Omega_A\Omega_B}{\omega_B^4},
\end{eqnarray*}
where
\[
\Omega_W=\sum\limits_{j=n-k+1}^n|\lambda_j(W)|,\quad \omega_W=\sum\limits_{j=1}^k|\lambda_j(W)|.
\]
Here, $\{\lambda_j(W)\}_{j=1}^n$ are the eigenvalues of a Hermitian matrix $W\in\mathbb{C}^{n\times n}$ with
\[
|\lambda_1(W)| \le |\lambda_2(W)|\le\cdots\le |\lambda_n(W)|.
\]

To illustrate  our theoretical results, we randomly generate the real symmetric matrices $A,B,C$, $\Delta A$, $\Delta B$, $\Delta C$,
by using the {\tt MATLAB} built-in functions {\tt rand}, {\tt randn}, {\tt orth}, {\tt diag} and {\tt ones}:
\begin{eqnarray*}
&& A={\tt rand}(n,n);\quad A=(A'+A)/2; \quad Q = {\tt orth}({\tt randn}(n,n)); \\
&& B = Q*{\tt diag}(50+\beta*(2*{\tt rand}(n,1)-{\tt ones}(n,1)))*Q';\quad B = (B'+B)/2;\\
&& C={\tt randn}(n,n);\quad C=(C'+C)/2;\\
&& \Delta A = \epsilon*(2*{\tt rand}(n,n) - {\tt ones}(n,n)); \Delta A = (\Delta A'+\Delta A)/2; \\
&&\Delta B = \epsilon*(2*{\tt rand}(n,n) - {\tt ones}(n,n)); \Delta B = (\Delta B'+\Delta B)/2; \\
&&\Delta C = \epsilon*(2*{\tt rand}(n,n) - {\tt ones}(n,n)); \Delta C = (\Delta C'+\Delta C)/2.
\end{eqnarray*}

For simplicity, we fix  $n=100$, $k=5$, and $\beta=10$.
Figure \ref{tr-fig} plots $\chi_*$, and the perturbation bounds $\xi_*$ and $\tau_*$ for varying $\epsilon$.
Figure \ref{tr-fig2} shows  $\hat{\chi}_*$ versus the error bounds $\hat{\xi}_*$ and $\hat{\tau}_*$ for different $\beta$ in terms of the SCF iterations.

\begin{figure}[ht]
\centering
\includegraphics[width=0.49\textwidth]{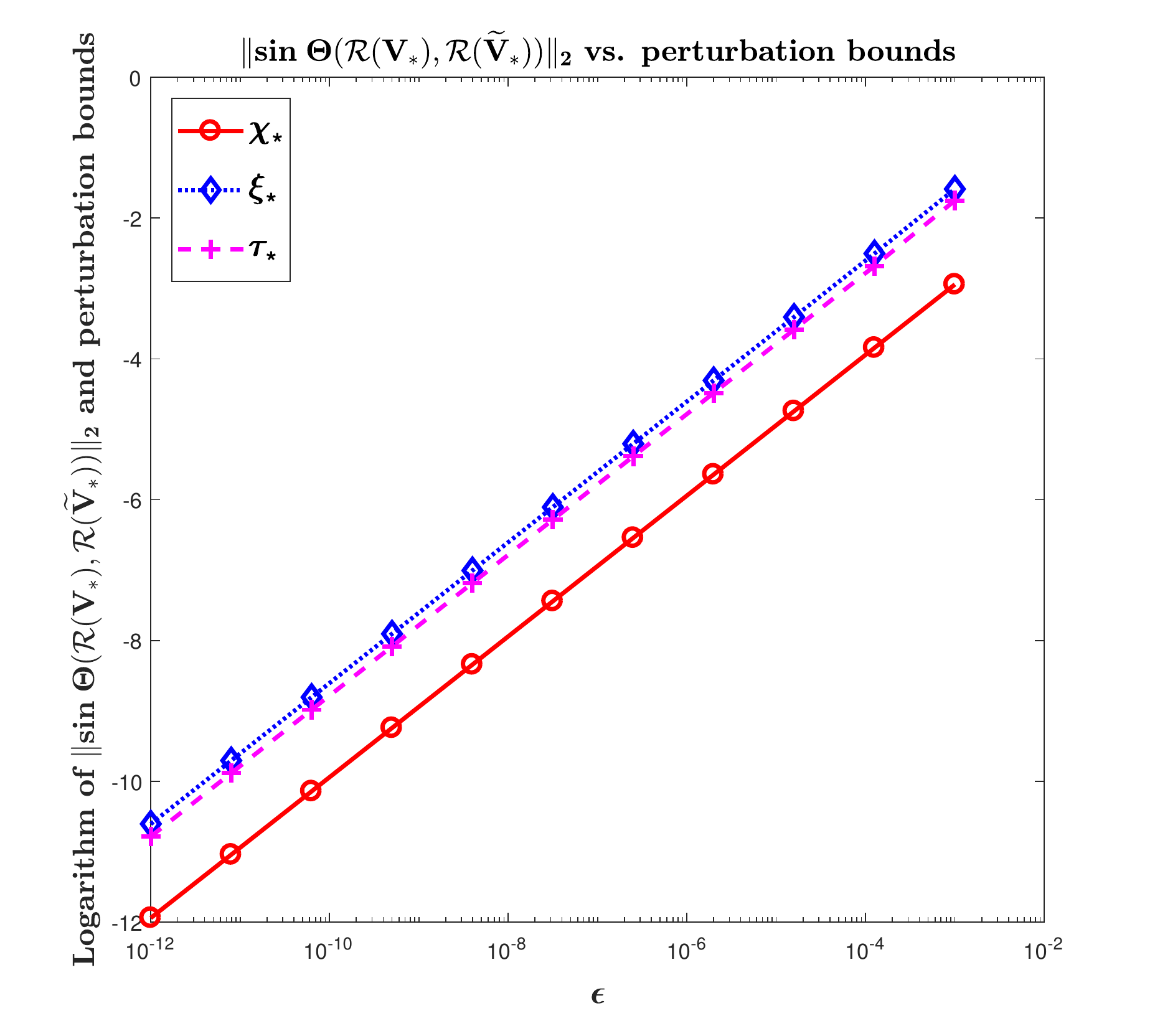}
\includegraphics[width=0.49\textwidth]{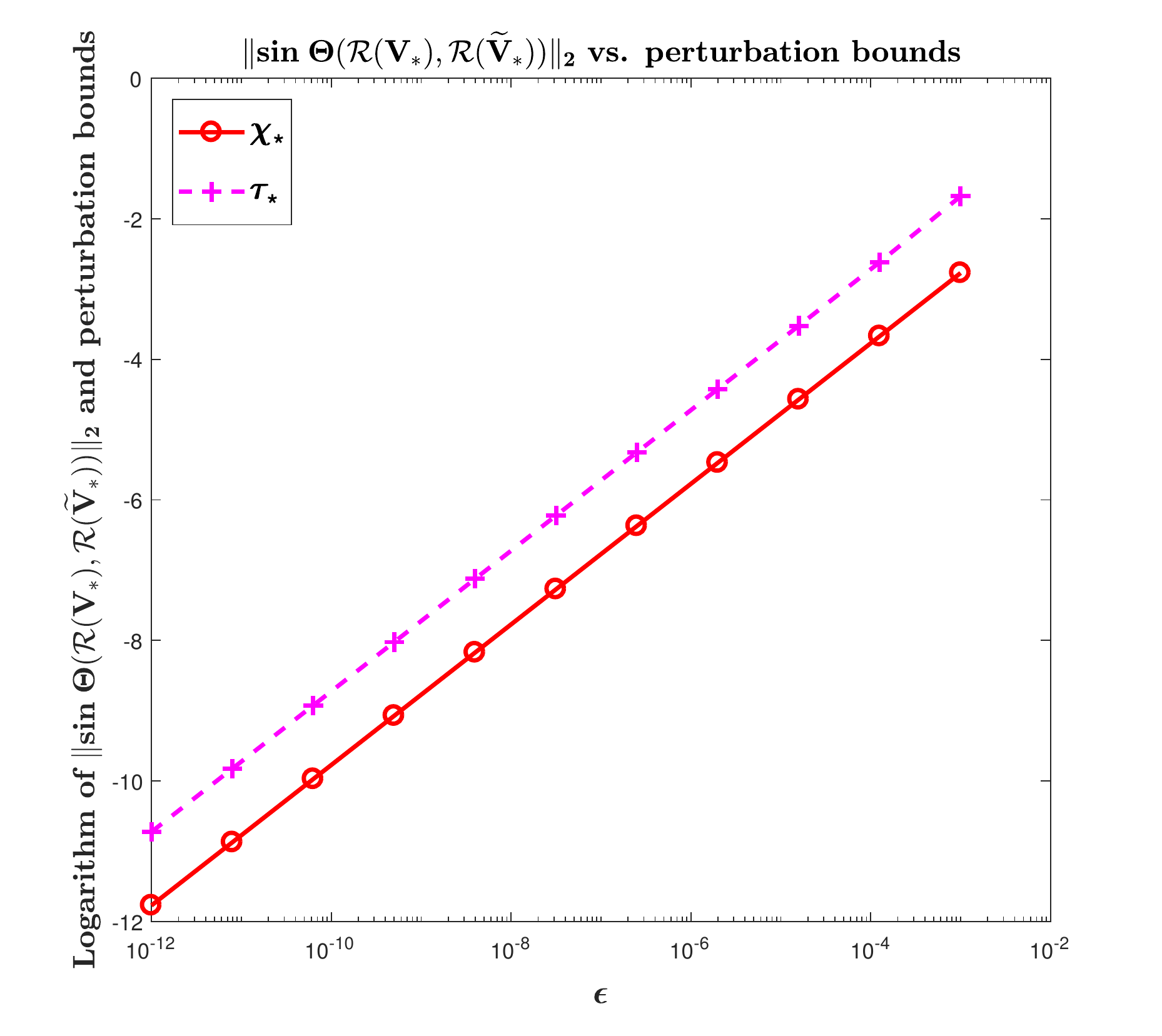}\\
        \small{$g/d=2.4152$ and $1/(g-d)=3.0297$}\hspace{.5in}
        \small{$g/d=1.9262$ and $1/(g-d)=4.9155$}\\
\includegraphics[width=0.49\textwidth]{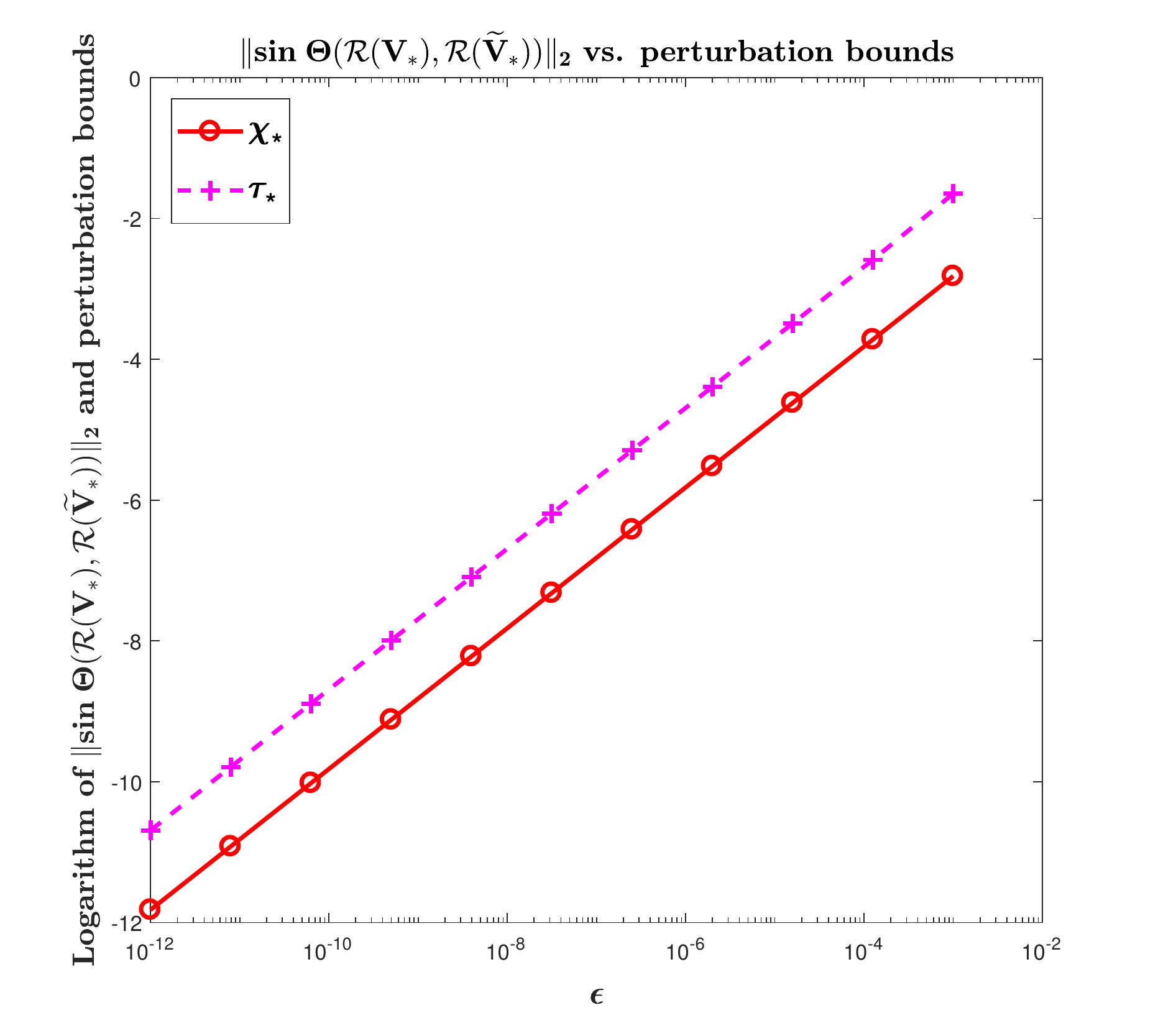}
\includegraphics[width=0.49\textwidth]{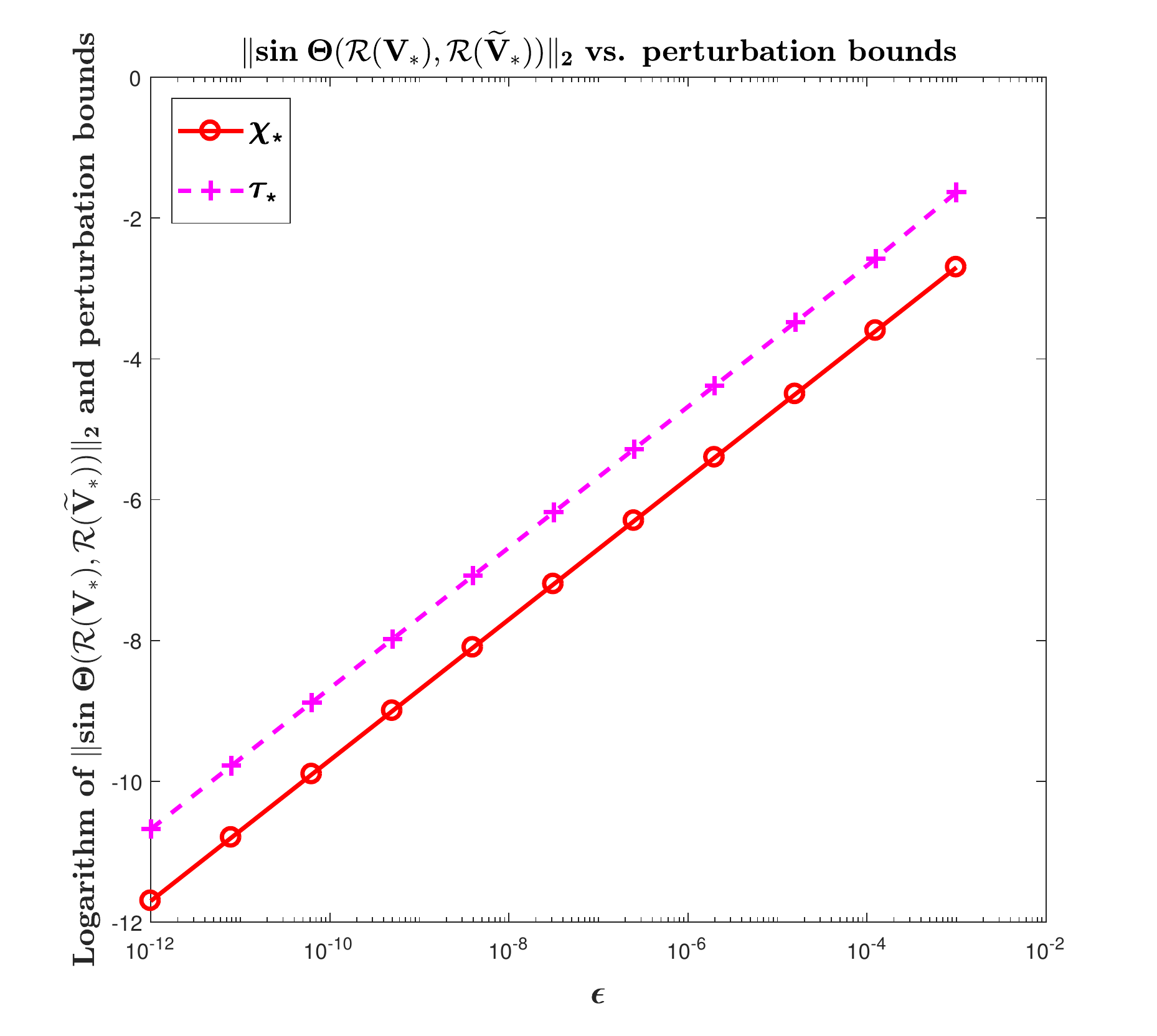}\\
        \small{$g/d=1.5351$ and $1/(g-d)=7.9778$}\hspace{.5in}
        \small{$g/d=1.0739$ and $1/(g-d)=60.229$}
\caption{$\|\sin\Theta(\mathcal{R}(V_*),\mathcal{R}(\wt{V}_*))\|_2$ vs. perturbation bounds for the trace ratio optimization}
 \label{tr-fig}
\end{figure}
\begin{figure}[ht]
\centering
\includegraphics[width=0.49\textwidth]{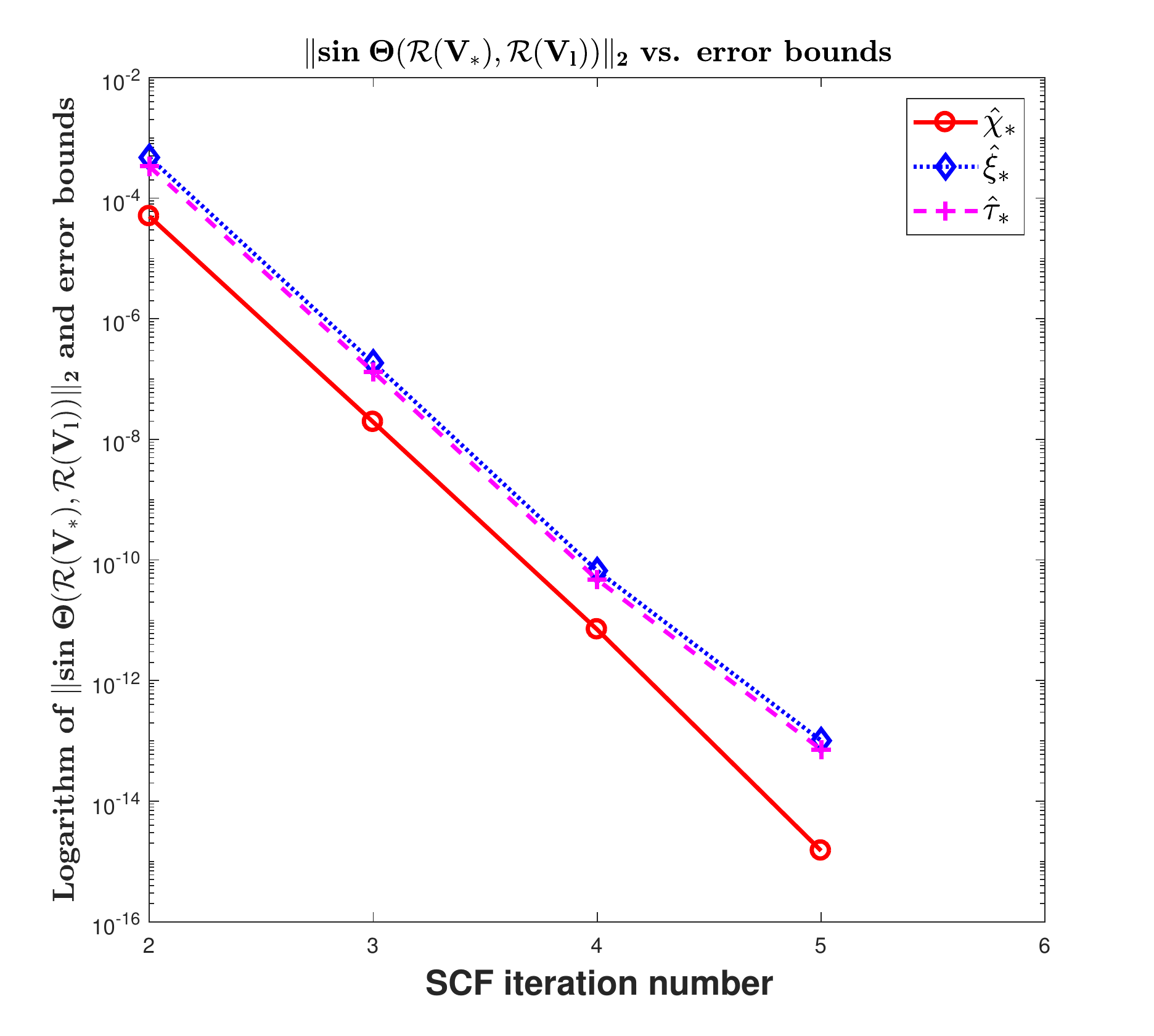}
\includegraphics[width=0.49\textwidth]{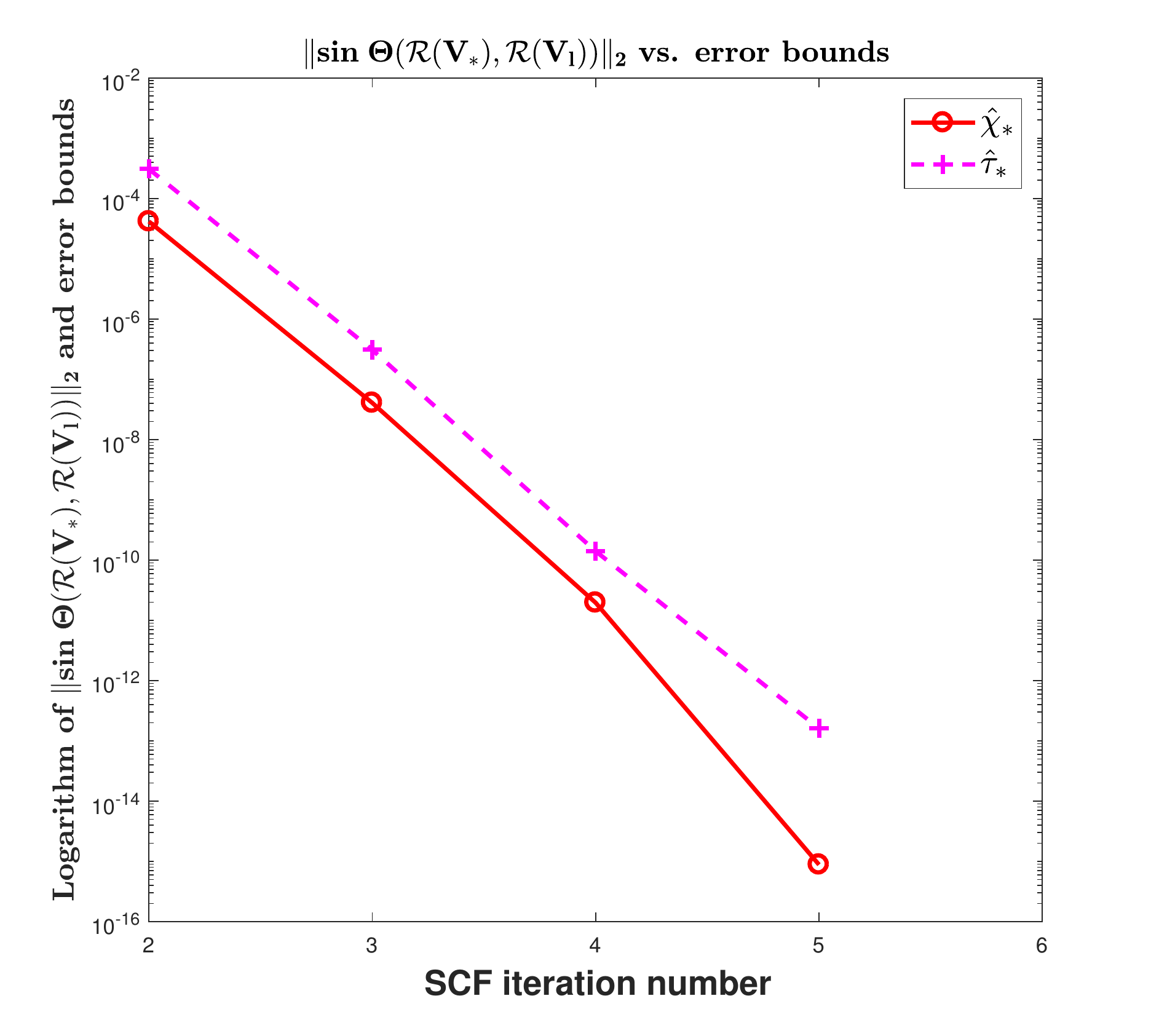}\\
        \small{$\beta=5$}\hspace{2.2in}
        \small{$\beta=10$}\\
\includegraphics[width=0.49\textwidth]{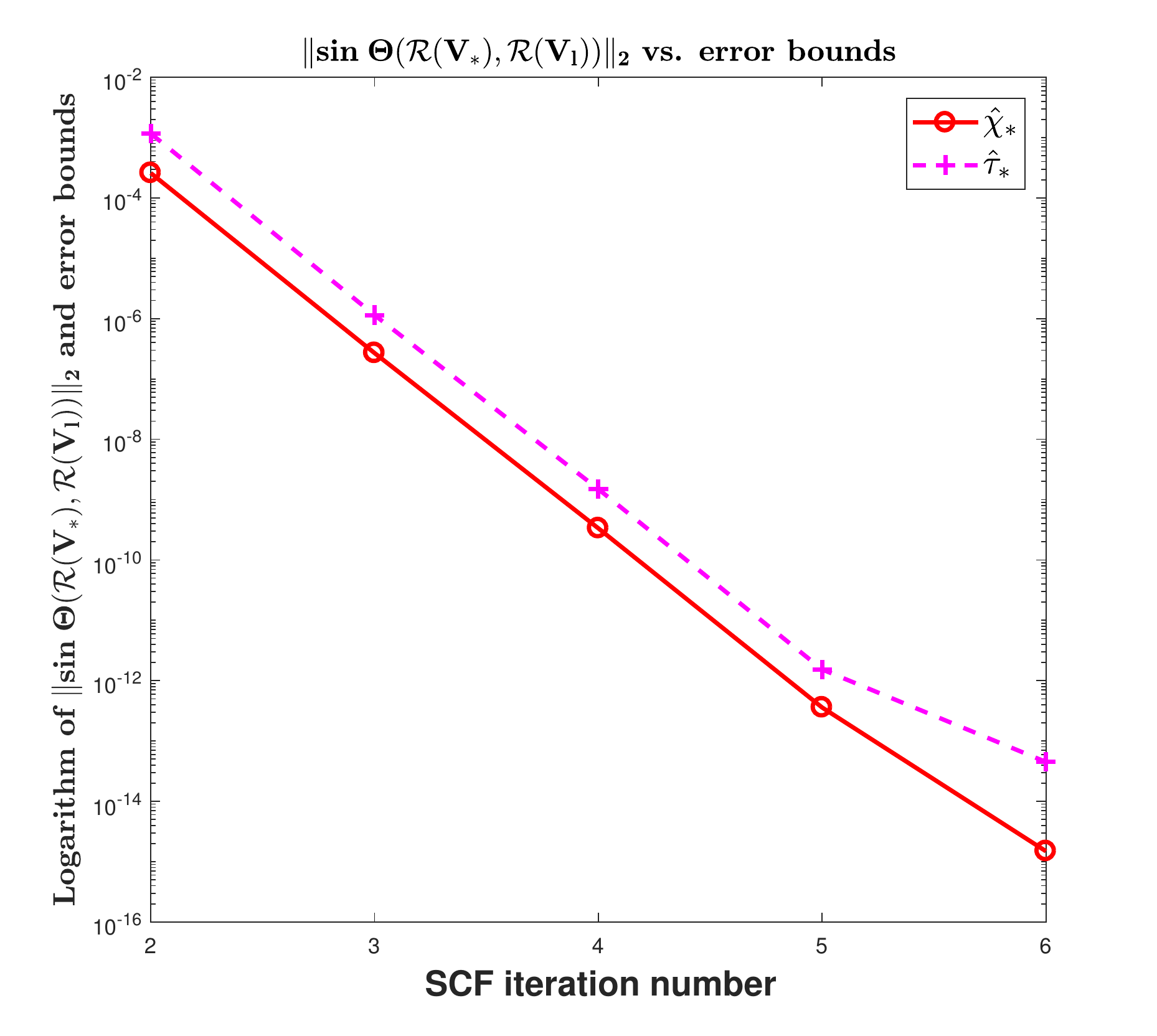}\\
       \small{$\beta=15$}
\caption{$\|\sin\Theta(\mathcal{R}(V_*), \mathcal{R}(V_l))\|_2$ vs. error bounds for the trace ratio optimization}
 \label{tr-fig2}
\end{figure}

We observe from Figure \ref{tr-fig} that when $\frac{g}{d}>2$,  both the assumptions of Theorem~\ref{thm1}  and Theorem~\ref{thm2} hold.
In this case, the perturbation bounds $\xi_*$ and $\tau_*$ are good upper bounds for the solution perturbation $\chi_*$ when $\delta$ is small,
while the perturbation bound  $\tau_*$ is slightly sharper than $\xi_*$ and $\gamma_*$.
However, when $1<\frac{g}{d}<2$, only the assumption of Theorem~\ref{thm2} holds.
In this case, the perturbation bound $\tau_*$ is good upper bounds for the solution perturbation $\chi_*$.
We have the similar observation for Figure \ref{tr-fig2} on $\hat{\chi}_*$ and the error bounds $\hat{\xi}_*$ and $\hat{\tau}_*$ in terms of the SCF iterations.

To further illustrate our theoretical results, in Table \ref{tr-table},
we report the estimated values of $\frac{g}{d}$ and $\frac{1}{g-d}$, the solution perturbation $\chi_*$, the perturbation bounds $\xi_*$, $\tau_*$, and $\gamma_*$ for fixed $\delta$ and varying  $\beta$,
where the symbol ``-" means the  upper bound $\xi_*$ is not a valid estimation value since the assumption  of Theorem~\ref{thm1} does not hold.
 Also,  Table \ref{tr-table2} displays the estimated values of $\frac{\hat{g}}{\hat{d}}$ and $\frac{1}{\hat{g}-\hat{d}}$, the solution perturbation $\hat{\chi}_*$, the error bounds $\hat{\xi}_*$, $\hat{\tau}_*$, and $\hat{\gamma}_*$ for  varying  $\beta$ in terms of the SCF iterations,
 where the symbol ``-" means the corresponding error bound  is not a valid estimation value since the assumption  of Corollary~\ref{cor1} or Corollary~\ref{cor2}  does not hold or the perturbation $\|R\|_2$ is not sufficiently small.

We see from Table \ref{tr-table} that, for a fixed $\delta$ and different $\beta$,
the estimated values of $\xi_*$, $\tau_*$, and $\gamma_*$ are valid upper bounds for the solution perturbation bound $\chi_*$.
We also see that  $\tau_*$ is shaper than $\xi_*$ and $\gamma_*$ and the assumption of Theorem~\ref{thm2} is weaker than that of Theorem~\ref{thm1}.
We have the similar observation for Table \ref{tr-table2} on  $\hat{\chi}_*$ and the error bounds $\hat{\xi}_*$, $\hat{\tau}_*$, and $\hat{\gamma}_*$ in terms of the SCF iterations.

\begin{table}[ht]\renewcommand{\arraystretch}{1.2} \addtolength{\tabcolsep}{2pt}
  \caption{Perturbation bounds for the trace ratio optimization}\label{tr-table}
  \begin{center} {\scriptsize
   \begin{tabular}[c]{|c|c|c|c|c|c|c|}
     \hline
\multicolumn{7}{|c|}{$\delta=10^{-12}$} \\  \hline
$\beta$& $g/d$ & $1/(g-d)$  & $\chi_*$ & $\xi_*$ & $\tau_*$  & $\gamma_*$ \\  \hline
$ 5$ & 2.7149e+00 & 3.9202e+00 & 1.0410e-12  & 3.2628e-11& 2.2968e-11 &  3.2628e-11 \\
$ 8$ & 2.1012e+00 & 4.7248e+00 &  1.0422e-12 & 3.9574e-11 &  2.0639e-11 &  3.9574e-11 \\
$10$ & 1.7617e+00 & 5.7274e+00 & 1.0383e-12 & -  &  1.9023e-11 &  4.8249e-11 \\
$12$ & 1.4442e+00 & 8.0504e+00 &  1.0387e-12 & - &  1.7211e-11 &  6.8344e-11 \\
$15$ & 1.0655e+00 & 4.0283e+01 & 1.0415e-12 & - &  1.4552e-11 &  3.4746e-10 \\
 \hline
\multicolumn{7}{|c|}{$\delta=10^{-6}$} \\  \hline
$\beta$& $g/d$ & $1/(g-d)$ & $\chi_*$  & $\xi_*$ & $\tau_*$& $\gamma_*$  \\  \hline
$ 5$ & 2.7149e+00 & 3.9202e+00 & 1.0407e-06  & 3.2630e-05 &  2.2972e-05 &  3.2630e-05 \\
$ 8$ & 2.1012e+00 & 4.7248e+00 & 1.0407e-06 &  3.9577e-05 &  2.0641e-05 &  3.9574e-05 \\
$10$ & 1.7617e+00 & 5.7274e+00 & 1.0407e-06  & - &  1.9024e-05 & 4.8254e-05 \\
$12$ & 1.4442e+00 & 8.0504e+00 &  1.0408e-06 &  - &  1.7212e-05 &  6.8344e-05 \\
$15$ & 1.0655e+00 & 4.0283e+01 & 1.0406e-06 &  - &  1.4552e-05 &  3.4746e-04 \\
 \hline
\multicolumn{7}{|c|}{$\delta=10^{-4}$} \\  \hline
$\beta$& $g/d$ & $1/(g-d)$ & $\chi_*$  & $\xi_*$ & $\tau_*$& $\gamma_*$  \\  \hline
$ 5$ & 2.7149e+00 & 3.9202e+00 & 1.0407e-04 & 3.2798e-03 &  2.3335e-03 &  3.2798e-03 \\
$ 8$ & 2.1012e+00 & 4.7248e+00 & 1.0407e-04 & 3.9876e-03 &  2.0865e-03 &  3.9574e-03 \\
$10$ & 1.7617e+00 & 5.7274e+00 & 1.0407e-04 & - &  1.9183e-03 &  4.8797e-03 \\
$12$ & 1.4442e+00 & 8.0504e+00 & 1.0407e-04 & - &  1.7318e-03 &  6.8344e-03 \\
$15$ & 1.0655e+00 & 4.0283e+01 & 1.0406e-04 &  -  & 1.4608e-03 &  3.4746e-02 \\
 \cline{2-7}\hline
\end{tabular} }
  \end{center}
\end{table}

\begin{table}[ht]\renewcommand{\arraystretch}{1.2} \addtolength{\tabcolsep}{2pt}
  \caption{Error bounds for the trace ratio optimization}\label{tr-table2}
  \begin{center} {\scriptsize
   \begin{tabular}[c]{|c|c|c|c|c|c|c|}
     \hline
\multicolumn{7}{|c|}{$\beta=5$} \\  \hline
$l$& $\hat{g}/\hat{d}$ & $1/(\hat{g}-\hat{d})$  & $\hat{\chi}_*$ & $\hat{\xi}_*$ & $\hat{\tau}_*$  & $\hat{\gamma}_*$ \\  \hline
$ 1$ & 3.4532e+00 & 2.7704e+00 & 9.9991e-01 & -                   & 3.1119e-01 & 2.0029e+01  \\
$ 2$ & 2.8587e+00 & 3.6565e+00 & 5.0006e-05 &  4.7307e-04 & 3.3702e-04 & 4.7272e-04 \\
$ 3$ & 2.8587e+00 & 3.6565e+00 & 1.9341e-08 &  1.8521e-07 & 1.3174e-07 & 1.8521e-07 \\
$ 4$ & 2.8587e+00 & 3.6565e+00 & 7.0187e-12 &  6.6451e-11 & 4.7267e-11 & 6.6451e-11 \\
$ 5$ & 2.8587e+00 & 3.6565e+00 & 1.5051e-15 &  1.0091e-13 & 7.1775e-14 &1.0091e-13  \\
 \hline
\multicolumn{7}{|c|}{$\beta =10$} \\  \hline
$l$& $\hat{g}/\hat{d}$ & $1/(\hat{g}-\hat{d})$  & $\hat{\chi}_*$ & $\hat{\xi}_*$ & $\hat{\tau}_*$  & $\hat{\gamma}_*$ \\  \hline
$ 1$ & 2.8375e+00 & 2.3391e+00 & 9.9992e-01 & -                  & 2.9621e-01 &1.8722e+01  \\
$ 2$ & 1.8076e+00 & 5.3220e+00 & 4.1495e-05 & -                  & 3.0839e-04 & 7.7035e-04  \\
$ 3$ & 1.8076e+00 & 5.3220e+00 & 4.0590e-08 & -                  & 3.0837e-07 & 7.7134e-07 \\
$ 4$ & 1.8076e+00 & 5.3220e+00 & 1.9616e-11 & -                  & 1.4031e-10 & 3.5097e-10  \\
$ 5$ & 1.8076e+00 & 5.3220e+00 & 8.8364e-16 & -                  & 1.6159e-13 & 4.0419e-13  \\
\hline
\multicolumn{7}{|c|}{$\beta =15$} \\  \hline
$l$& $\hat{g}/\hat{d}$ & $1/(\hat{g}-\hat{d})$  & $\hat{\chi}_*$ & $\hat{\xi}_*$ & $\hat{\tau}_*$  & $\hat{\gamma}_*$ \\  \hline
$ 1$ & 8.3302e-01 &-1.6520e+01 &  9.9901e-01 &        -            & -                     & -   \\
$ 2$ & 1.1592e+00 & 1.7326e+01 & 2.5827e-04 &       -             & 1.1659e-03 & 1.2077e-02  \\
$ 3$ & 1.1592e+00 & 1.7326e+01 & 2.6832e-07 &        -            & 1.1450e-06 & 1.1899e-05  \\
$ 4$ & 1.1592e+00 & 1.7326e+01 & 3.3430e-10 &        -            & 1.4918e-09 & 1.5502e-08  \\
$ 5$ & 1.1592e+00 & 1.7326e+01 & 3.5858e-13 &         -           & 1.5329e-12 & 1.5929e-11  \\
$ 6$ & 1.1592e+00 & 1.7326e+01 & 1.4886e-15 &         -           & 4.4921e-14 & 4.6680e-13  \\
 \cline{2-7}\hline
\end{tabular} }
  \end{center}
\end{table}


\section{Conclusion}\label{sec:conclusion}
In this paper, we have studied the perturbation theory of NEPv \eqref{eq:nep}.
Two perturbation bounds are established,
based on which the condition number for the NEPv can be introduced.
Furthermore, two computable error bounds are also obtained.
Theoretical results are applied to the KS equation and the trace ratio problem.
Numerical results show that both the perturbation bounds and the error bounds are fairly sharp,
especially the perturbation bound in Theorem ~\ref{thm2} and the error bound in Corollary~\ref{cor2}.

\bibliography{perturbation_nep.bbl}
\end{document}